\documentclass[oneside,notitlepage,12pt]{article}

\usepackage{amssymb}
\usepackage[leqno]{amsmath}
\usepackage{amsfonts}
\usepackage{amsopn}
\usepackage{amstext}
\usepackage{amsthm}
\usepackage[normalem]{ulem}

\usepackage{tikz}
\usetikzlibrary{arrows.meta}
\usetikzlibrary{graphs}
\usetikzlibrary{decorations.pathmorphing}

\usepackage{enumitem}

\usepackage{verbatim}
\usepackage[colorlinks, backref]{hyperref}

\usepackage{mathrsfs}
\usepackage{calrsfs}
\usepackage{clock} 

\usepackage{marvosym} 

\setlist{itemsep = 0pt}
\setlist[enumerate, 1]{label=\upshape (\arabic*), ref=(\arabic*)}

\newtheorem{tw}{Theorem}[section]
\newtheorem{wn}[tw]{Corollary}
\newtheorem{lm}[tw]{Lemma}
\newtheorem{prop}[tw]{Proposition}

\theoremstyle{definition}
\newtheorem{df}[tw]{Definition}
\newtheorem{ex}[tw]{Example}

\newtheorem{question}[tw]{Question}
\theoremstyle{remark}
\newtheorem{uwgi}[tw]{Remark}
\newtheorem{observation}[tw]{Observation}

\newcommand \set [1]{\{#1\}}
\newcommand \seq [1]{\langle #1 \rangle}
\newcommand \map [3]{#1\colon #2 \to #3} 
\newcommand \maps {\colon} 
\newcommand \im [1]{[#1]} 
\newcommand \preim [1]{^{-1}[#1]} 
\newcommand \fiber [1] {^{-1}(#1)}
\newcommand \from {\leftarrow} 
\newcommand \restr [1] {\mathord{\upharpoonright}_{#1}}
\newcommand \id [1]{{\operatorname{id}_{#1}}}
\newcommand \dom {\operatorname{dom}}
\newcommand \rng {\operatorname{rng}}

\newcommand \cmp {\circ} 
\newcommand \subs {\subseteq}
\newcommand \sups {\supseteq}
\renewcommand \iff {\Longleftrightarrow}
\newcommand \meet {\wedge}

\newcommand \abs [1] {\lvert#1\rvert}
\newcommand \card [1] {\lvert#1\rvert} 
\newcommand \clo [1] {\overline{#1}} 

\newcommand \diam {\operatorname{diam}}
\newcommand \supp {\operatorname{supp}}

\newcommand \sig {\sigma}
\newcommand \eps {\varepsilon}

\newcommand \N {\mathbb{N}}
\newcommand \Z {\mathbb{Z}}
\newcommand \Q {\mathbb{Q}}
\newcommand \R {\mathbb{R}}
\newcommand \Qpos {\Q^{\geq 0}}
\newcommand \Rpos {\R^{\geq 0}}
\let \qplus \Qpos

\newcommand \U {\mathbb{U}}

\newcommand \C {\mathscr{C}}
\newcommand \D {\mathscr{D}}
\newcommand \K {\mathscr{K}}
\newcommand \T {\mathscr{T}}

\newcommand \fin {{\operatorname{fin}}}

\newcommand \Lin {\mathfrak{L}}
\newcommand \Ult {\mathfrak{U}}
\newcommand \UltIso {\mathfrak{I}}
\newcommand \UltConv {\Ult^\prec}
\newcommand \IsoConv {\mathfrak{C}} 
\newcommand \Tree {\mathfrak{T}}
\newcommand \LevTree {\mathfrak{TL}}
\newcommand \ATree {\mathfrak{A}}
\newcommand \Ball {\mathfrak{B}}

\newcommand \ultras {\mathfrak U}
\newcommand \ultraf {\ultras_{\operatorname{fin}}}

\newcommand \ultracon {\ultras^\prec}
\newcommand \ultraconf {\ultracon_{\operatorname{fin}}}

\newcommand \Aut {\operatorname{Aut}}
\newcommand \Iso {\operatorname{Iso}}
\newcommand \AutM {\operatorname{Aut_m}}
\newcommand \IsoM {\operatorname{Iso_m}}
\let \aut \Aut

\newcommand \lev {\ell}
\newcommand \Lev {\operatorname{Lev}}

\newcommand \Bee {\mathscr{B}}
\newcommand \Beebar {\mkern3mu\overline{\mkern-3mu\Bee}}
\newcommand \Max {\operatorname{Max}}

\newcommand \monster {(\UltIso^{\Qpos}_{\rm prec})_\fin}
\newcommand \gen {{\operatorname{gen}}}

\newcommand \email [1] {%
    \quad{\Letter\,\small\href{mailto:#1}{\nolinkurl{#1}}}%
}

\newcommand \arxivlink [1]{%
    \href{https://arxiv.org/abs/#1}{\texttt{arXiv:#1}}%
}

\hypersetup{
	pdftitle = {Universal homogeneous two-sorted ultrametric spaces},
	pdfauthor = {Adam Bartoš, Wiesław Kubiś, Aleksandra Kwiatkowska, Maciej Malicki},
	colorlinks,
	linkcolor = [rgb]{0, 0, 0.7},
	citecolor = [rgb]{0, 0.7, 0},
	urlcolor = [rgb]{0.4, 0.2, 0},
	linktoc = section,
}

\title{Universal homogeneous \\ two-sorted ultrametric spaces}
\author{
{\sc Adam Barto\v{s}}
    \email{bartos@math.cas.cz}\\
    {\small Institute of Mathematics, Czech Academy of Sciences (CZECHIA)}
\and
{\sc Wies{\l}aw Kubi\'s}
    \email{kubis@math.cas.cz}\\
    {\small Institute of Mathematics, Czech Academy of Sciences (CZECHIA)}
\and
{\sc Aleksandra Kwiatkowska}
    \email{aleksandra.kwiatkowska@math.uni.wroc.pl}\\
    {\small University of Wrocław (POLAND) and  University of M\"unster (GERMANY)}
\and
{\sc Maciej Malicki}
    \email{mmalicki@mimuw.edu.pl}\\
    {\small Faculty of Mathematics, Informatics and Mechanics, University of Warsaw (POLAND)}
}

\date{\clocktime\today}

\begin{document}

\maketitle

\begin{abstract}
    We view ultrametric spaces as two-sorted structures consisting of a set of points and of a linearly ordered set of distances.
    We call the appropriate notion of embeddings distance-carrying (dc- for short). Those are obtained by combining isometries and linear order embeddings.
    We show that the class of all finite two-sorted ultrametric spaces with dc-embeddings is Fraïssé, and that the limit is the countable rational Urysohn ultrametric space $\U$.
    The space $\U$ is dc-universal for all countable ultrametric spaces, and its Cauchy completion $\clo{\U}$ is dc-universal for all separable ultrametric spaces, which is in contrast with the situation of classical ultrametric spaces and isometric embeddings, where no such universal space can exist.
    
    We study further properties of $\U$, of its variants, and of its automorphism group, which is richer than its group of isometries. In particular, we provide two types of tree representations of the two-sorted  ultrametric spaces, discuss connections to valued fields, and characterize the automorphism group of $\U$ as the semidirect product of a group of order preserving bijections and a group of isometries.
    Furthermore, we show universality of $\Aut(\U)$ and identify its universal minimal flow.
\end{abstract}



\vspace{-2ex}

\vspace{-2ex}
\paragraph{Acknowledgments.} 
Research of Adam Bartoš and Wiesław Kubiś was supported by GAČR (Czech Science Foundation) grant EXPRO 20-31529X and by the Czech Academy of Sciences (RVO 67985840). Research of Aleksandra Kwiatkowska was supported by DFG (German Research Foundation) under Germany’s Excellence Strategy EXC 2044–390685587, Mathematics M\"{u}nster: Dynamics–Geometry–Structure. Research of Maciej Malicki was partially supported by the National Science Centre, Poland
under the Weave-UNISONO call in the Weave programme [grant no
2021/03/Y/ST1/00072].


\tableofcontents

\section{Introduction}

A metric space $(X,d)$ is called an {\em ultrametric space} if for all $x,y,z\in X$, the ultrametric triangle inequality holds: $d(x,y)\leq\max\{d(x,z), d(y,z)\}$. 
Consequently, for any $x,y,z\in X$, either $d(x,y)=d(x,z)= d(y,z)$ (the triangle is equilateral), 
or, if we assume without loss of generality that $d(x,y)$ is the minimum, it holds that $d(x,y)<d(x,z)= d(y,z)$.

The {\em (open) ball} of center $x\in X$ and radius $r>0$ in an ultrametric space is the set $B_r(x)=\{y\in X\colon d(y,x)<r\}$. Analogously, we define the \emph{closed ball} as $\bar{B}_r(x)=\{y\in X\colon d(y,x)\leq r\}$. Note that for any two balls $B$ and $C$, they are either disjoint ($B\cap C=\emptyset$) or one is contained in the other ($B\subseteq C$ or $C\subseteq B$). Moreover, every point in a ball can serve as its center.

\begin{ex}
\begin{enumerate}
\item The Baire space $\N^\N=\{x\colon \N\to \N\}$ equipped with the metric $d(x,y)=\max\{2^{-n}\colon x(n)\neq y(n)\}$, for $x\neq y$, is an ultrametric space.
\item A valued field (with a countable value group) gives rise to an ultrametric space. For example, consider the $p$-adic numbers $\Q_p$, where $p$ is a prime. Any rational number $x\neq 0$ can be uniquely represented as $\frac{p^k r}{s}$, where $k,r,s\in \Z$, $s > 0$, and $r,s$ are not divisible by $p$. We define the norm $|x|_p=p^{-k}$ and the distance $d(x,y)=|x-y|_p$, letting $\Q_p$ be the completion of the rationals with respect to this metric $d$.
\end{enumerate}
\end{ex}

We are interested in \emph{universal} ultrametric spaces. In the context of all metric spaces, there exists a separable metric space into which every separable metric space embeds isometrically.
Standard examples include the space of continuous real-valued functions $C([0,1])$ equipped with the supremum norm or the Urysohn metric space, the latter being also \emph{(ultra)homogeneous}, i.e. every isometry between its finite subspaces extends to an isometry of the whole space.

For ultrametric spaces, however, the situation is diametrically different: there is no separable ultrametric space that is universal for all separable, or even just finite, ultrametric spaces. In 1994, Vestfrid~\cite{V} constructed an ultrametric space universal for all separable ultrametric spaces, but that space is not separable. It is defined as
\[ U=\{(x_n)\in (\Rpos)^\mathbb{N}\colon x_{n+1}\leq x_n \text{ for all } n
\text{ and } (x_n) \text{ converges to } 0 \}, \]
where for $x\neq y$, we let $d( x, y)= \max\{x_k,y_k\}$, with $k$ being the least index such that $x_k\neq y_k$.

Furthermore, A. J. Lemin and  V. A. Lemin \cite{LL}  proved that there is no ultrametric space of weight $<\mathfrak{c}$ into which every separable ultrametric space embeds isometrically. The authors also provided an example of a space with weight $\mathfrak{c}$:
$$ L_\N=\{f\colon \Q^{> 0}\to \N\colon f {\rm{ \ is\ eventually\ }} 0\}$$ 
equipped with the distance
\[ \textstyle
    d(f,g)=\sup_{\Rpos}\{x\in \Q^{> 0}\colon f(x)\neq g(x)\}
\]
for $f\neq g$. The space $L_\N$ is complete and universal for ultrametric spaces of cardinality $<\mathfrak{c}$. 
Let us just mention that replacing $\N$ with a cardinal number $\kappa$ in the definition above,
Lemins obtained the space $L_\kappa$, which is universal for all ultrametric spaces of weight $\leq \kappa$, and  $\card{L_\kappa}=w(L_\kappa)=\kappa^{\aleph_0}$.

If we restrict ourselves to a countable set of distances $0 \in E \subseteq \Rpos$, a universal space does exist. In that case,
\[
    U_E =\{(x_n)\in E^\mathbb{N}\colon x_{n+1}\leq x_n \text{ for all } n
    \text{ and } (x_n) \text{ converges to } 0 \}
\]
equipped with $d(x, y)= \max\{x_k,y_k\}$, with $k$ being the least index such that $x_k\neq y_k$, for $x\neq y$,
is a separable, complete $E$-valued ultrametric space into which all $E$-valued ultrametric spaces embed isometrically, as shown in the work of Gao and Shao \cite{GS}. This construction was inspired by that of Vestfrid.
All these examples admit a unified treatment, as sketched in Remark~\ref{rmk:representations}.

The reason why there is no universal separable ultrametric space $X$ lies in too strict treatment of distances: for every countable dense subspace $D \subseteq X$ and two points $x \neq y \in X$ there are points $x', y' \in D$ with $d(x, x'), d(y, y') < d(x, y)$.
Hence, by the ultrametric triangle inequality we have $d(x, y) = d(x, y') = d(x', y')$, and so $d(x, y)$ is one of the countably many distances attained by $D$.
Therefore, $X$ does not contain even all two-point spaces, and if we wish to find a universal separable ultrametric space, we must disregard exact numerical values for distances.

We work with ultrametric spaces viewed as two-sorted structures where the metric takes values in a linearly ordered set of distances.
We also consider appropriate two-sorted embeddings, called \emph{dc-embeddings}, preserving the orders of the distance sets.
Using Fraïssé theory, we will construct such a two-sorted ultrametric space. The Fraïssé limit $\U$ will be countable and universal for all countable two-sorted ultrametric spaces.
Furthermore, since a (uniformly continuous) dc-embedding from one ultrametric space to another extends uniquely to a dc-embedding between their Cauchy completions, the completion $\clo{\U}$ will be universal for all separable two-sorted ultrametric spaces.
More broadly, in the paper we embark on a systematic study of two-sorted ultrametric spaces, ways of representing them, methods of constructing universal objects, and properties of their automorphism groups.

In a recent related work by Gheysens, Pavlica, Pech, Pech, and Schneider~\cite{GPPPS_Echeloned}, the authors consider structures with a quaternary relation expressing when points in a pair are closer to each other than points in another pair, without explicitly referring to any distances.
They call such structures \emph{echeloned spaces}, and among other results prove that finite echeloned spaces form a Fraïssé class, that their order expansions form a Ramsey class, and that the automorphism group of the Fraïssé limit is universal (via a Katětov functor).
Even though echeloned spaces are substantially different from ultrametric spaces, some of the results are parallel to ours.
We compare both structures in more detail in Remark~\ref{rmk:echeloned}.

In a follow-up paper \cite{paper2} we prove the existence of a generic dc-automorphism of~$\U$ and the non-existence of a generic pair of dc-automorphisms. We achieve the former result by showing cofinal amalgamation property of partial automorphisms and characterizing amalgamation bases. Moreover, we develop a general strategy for showing cofinal amalgamation property for a broad class of categories. Additionally, we carry out a detailed analysis of possible structures of an orbit of a point under a dc-automorphism.

\subsection{Summary of the results}

Initially motivated by the study of universal ultrametric spaces, this work provides a construction of a generic universal ultrametric space and  investigates  its completions and their automorphism groups.
Parallel to these results, we develop several tools of independent interest such as principles for transferring Fraïssé-theoretic properties between categories, and representations of ultrametric spaces via certain categories of trees and ball spaces and via spaces of functions.
While this is partially present in the mathematical folklore, we significantly extend the tools and provide a systematic and unified treatment.
We now summarize the variety of obtained results for the reader's convenience.

Section~\ref{Sec:prelim} establishes the general framework for studying two-sorted ultrametric spaces, including two ways to represent them as trees, and for studying their automorphism groups. In Section~\ref{Sec:Fraisse} we construct the universal homogeneous two-sorted ultrametric space $\U$ and several related Fraïssé limits, and we investigate its completions. Section~\ref{Sec:automorphism} compares the automorphism group $\Aut(\U)$ with the isometry group $\Iso(\U)$ and discusses universality of $\Aut(\U)$ via Katětov functors.

\medskip\noindent
\textbf{Preliminaries (Section~\ref{Sec:prelim}).}
We introduce all key notions, present our setup, and collect general observations for later use.

\begin{itemize}
    \item Despite working with concrete possibly many-sorted structures, we use a basic category-theoretic framework to systematically organize and compare classes of structures we are interested in.
    This includes notation for induced categories and of finite structures and of partial automorphisms.
    
    \item Ultrametric spaces are introduced as two-sorted structures, where the first sort consists of points and the second of distances in a linear order. Maps, called dc-embeddings, are not necessarily isometric embeddings, we only require that they preserve the order between distances.
    This does not always imply (uniform) continuity.
    We provide a characterization (Lemma~\ref{thm:dc-continuity}).
    
    \item The Cauchy completion $\clo{X}$ of a two-sorted ultrametric space $X$ behaves as expected, and in particular has the universal property for uniformly continuous dc-embeddings (Observation~\ref{exttoCauchy}).

    \item For any ultrametric space $X$, we endow $\Aut(X)$ with two topologies of pointwise convergence depending on whether we view $X$ as a discrete or metric structure -- the latter is denoted by $\AutM(X)$.
    If $X$ is separable without isolated points, then the adjacent balls functor $\Bee$ (Proposition~\ref{prop:thebee}) induces an isomorphism of Polish groups $\AutM(\clo{X}) \to \Aut(\Bee_X)$ (Proposition~\ref{prop: beta}).
\end{itemize}

\medskip\noindent
\textbf{Generic ultrametric spaces (Section~\ref{Sec:Fraisse}).} We construct the universal homogeneous two-sorted ultrametric space $\U$ as a Fraïssé limit of the class of finite two-sorted ultrametric spaces.

\begin{itemize}
\item We establish transfer principles (Propositions~\ref{thm:fin_transfer} and~\ref{thm:Flim_transfer}) that allow us to transfer
Fraïssé-theoretic properties between different categories of structures, which we then apply to several classes related to ultrametric spaces.

\item We show that the class $\Ult_\fin$ of all finite two-sorted ultrametric spaces is a Fraïssé class (Lemma~\ref{thm:one_point_amalgamation}). Its Fraïssé limit $\U$ is dc-isomorphic to $\U_\Q$, the countable rational Urysohn ultrametric space, and so is both dc-homogeneous and iso-homogeneous (Theorem~\ref{thm:U_Fraisse}).

\item The countable rational Urysohn ultrametric space $\U$ is universal for countable ultrametric spaces even with respect to uniformly continuous dc-embeddings (Proposition~\ref{thm:universal_separable}), and therefore its Cauchy completion $\clo{\U}$ is universal for separable ultrametric spaces (Corollary~\ref{wn:universal u bar}).
\item 
    The class $\Ball_\fin$ of all finite leveled adjacency trees is Fraïssé, and its limits is the open ball structure $\Bee_\U$ of the space $\U$ (Proposition~\ref{thm:ball_Fraisse_limit}).
    Moreover, the Polish group $\Aut(\Bee_\U)$ is isomorphic to $\AutM(\clo{\U})$ (Proposition~\ref{thm:ball_automorphism_group}).

\item We study convexly ordered ultrametric spaces, i.e. linearly ordered in a way that each ball is convex. The resulting class $\Ult^\prec_\fin$ is a  Fraïssé class whose limit is $\U$ endowed with the generic convex order (Theorem~\ref{thm:CU_Fraisse}).

\item For $M \subseteq \N^{\geq 2} \cup \{\infty\}$, we consider ultrametric spaces where  sizes of maximal equilateral sets are bounded by elements of $M$. The resulting class $\Ult^M_\fin$ is Fraïssé and its limit $\U^M$ is again both dc-homogeneous and iso-homogeneous (Theorem~\ref{thm:MU_Fraisse}).

    \item We provide concrete representations of $\U$ and of its Cauchy and spherical completions (Section~\ref{section:concrete}), and we summarize their homogeneity properties (Proposition~\ref{thm:concrere}).
    The concrete representation of $\U$ will also give us a section for the canonical projection $\pi\maps \Aut(\U) \to \Aut(\Qpos)$, and similarly for the completions (Proposition~\ref{thm:section}).
 
\item  We discuss connections to valued fields, 
    in particular which of the concrete representations admit a natural field structure (Corollary~\ref{thm:correspondeces}).
    We conclude that $\U$ is the ultrametric reduct of every valued field of the form $k(t^\Q)$ for a countable field $k$ (Corollary~\ref{wn:Urysohn reduct}).
  
\end{itemize}

\medskip\noindent
\textbf{Automorphism groups (Section~\ref{Sec:automorphism}).} 
We investigate the relationship between the automorphism group $\Aut(\U)$, the isometry group $\Iso(\U)$, and the automorphism group $\Aut(\Qpos)$ of the ordered set of non-negative rationals.

\begin{itemize}
    \item We realize the dc-automorphism groups as topological semidirect products, namely $\Aut(\U) \cong \Iso(\U) \rtimes \Aut(\Qpos)$ and  $\AutM(\clo{\U}) \cong \IsoM(\clo{\U}) \rtimes \Aut(\Qpos)$ (Theorem~\ref{thm:semidirect}).
    
    \item Using results of Jahel--Zucker for short exact sequences, we establish extreme amenability of $\aut(\U^\prec)$ and amenability of $\aut(\U)$ (Theorem \ref{examen}).
    Furthermore, using work of Kechris--Pestov--Todorčević, we describe the universal minimal flow of $\aut(\U)$ (Theorem \ref{tw: umf-desrip}).
    
    \item Finite ultrametric spaces admit a Katětov functor (Proposition~\ref{thm:Katetov}), and consequently $\Aut(\U)$ is a topologically universal automorphism group for countable ultrametric spaces (Corollary~\ref{wn-katetov}).
    Universality of $\AutM(\clo{\U})$ for automorphism groups of separable ultrametric spaces remains open (Question~\ref{que:complete_aut_universal}). 
\end{itemize}

\section{Preliminaries}\label{Sec:prelim}

\subsection{General framework}

As we shall discuss several different (but related) types of structures as well as several types of structure-preserving maps, we shall adopt the language of category theory to treat the matter in a systematic way. 
Recall that a category $\C$ consists of a family of objects (sometimes called $\C$-objects) and of a composable family of morphisms (sometimes called arrows or $\C$-maps),
i.e. for every $\C$-maps between $\C$-objects $f\maps X \to Y$ and $g\maps Y \to Z$, we have the composition $g \cmp f\maps X \to Z$.
The composition is associative.
We also have the identity morphism $\id{X}\maps X \to X$ for every $\C$-object $X$ that is neutral with respect to composition.

Recall that a \emph{subcategory} $\D \subseteq \C$ is a category consisting of some $\C$-objects and some $\C$-maps between them, with the same composition.
A subcategory $\D \subseteq \C$ is called \emph{full} if we restrict only objects, i.e. if every $\C$-map between $\D$-objects is a $\D$-map, and a subcategory $\D \subseteq \C$ is called \emph{wide} if it has the same objects as $\C$, i.e. if we are restricting only the morphisms.
We will use some other standard category-theoretic notions such as functors, natural transformations, equivalences of categories, and colimits. See a standard reference such as \cite{MacLane} for details.

Our categories consist of model-theoretic structures in a fixed signature (sometimes many-sorted, but we use only finitely many sorts), and of all embeddings.
We call them \emph{categories of structures}.
For every category of structures $\C$, the full subcategory of all finite structures is denoted by $\C_\fin$.
Also, there is a notion of a substructure, i.e. of an embedding that is an inclusion on the level of sets.
The notation is $X \leq Y$, meaning that $X$ is a substructure of $Y$ and that $Y$ is an extension of $X$.
Arguments are sometimes simplified by replacing embeddings with substructures. This is without loss of generality because, up to isomorphism, every embedding is an inclusion.

Recall that in general an {\it isomorphism} in a category $\C$ is a $\C$-map $f\maps X \to Y$ such that there is $\C$-map $f^{-1}\maps Y \to X$ satisfying $f^{-1} \cmp f = \id{X}$ and $f \cmp f^{-1} = \id{Y}$.
In a category of structures, it is equivalently a $\C$-map that is bijective (in every sort) as $\C$-maps are embeddings.
Isomorphisms $X \to X$ are called {\it automorphisms}.
For every $\C$-object $X$ the automorphism group of $X$ is denoted by $\Aut(X)$.

\begin{uwgi} \label{rem:many-sorted}
    Some general results we will use, such as the characterization of existence of a generic automorphism due to Kerchis--Rosendal~\cite{KR} or the Kechris--Pestov--Todorčević correspondence~\cite{KPT}, are stated for classes of one-sorted structures (with both functions and relations).
    This is not a problem for us as a many-sorted structure can be turned into a one-sorted structure while preserving the notion of embedding, and hence any category of structures is equivalent to a category of one-sorted structures.
\end{uwgi}

Let $\C$ be a category of finite structures.
For every $\C$-object $A$ a \emph{partial automorphism} of $A$ is a $\C$-isomorphism $p$ such that $\dom(p), \rng(p) \leq A$, i.e. an isomorphism between substructures of $A$.
An embedding of partial automorphisms $f\maps \seq{A,p} \to \seq{B,q}$ is a $\C$-map $f\maps A \to B$ such that $f \cmp p = q \cmp f$ on $\dom(p)$, from which it follows that $f$ restricts to (unique) $\C$-maps $\dom(p) \to \dom(q)$ and $\rng(p) \to \rng(q)$.
We denote the category of partial automorphisms in $\C$ and their embeddings by $\C^*$.
We can view $\C^*$ as a category of structures. Indeed, if structures in $\C$ are one-sorted, we turn $\seq{A,p}$ into a three-sorted structure with sorts $A,\dom(p),\rng(p)$, and we add the embeddings  $\dom(p)\to A$ and $\rng(p) \to A$ as well as both $p$ and $p^{-1}$ as operations. If $\C$ is many-sorted, we do this for every sort.

\subsection{Two-sorted ultrametric spaces}\label{prel:2sort}

By $\Lin$ we denote the category of all linear orders $\seq{L, \leq_L}$ and all embeddings, i.e. maps $f\maps L \to L'$ such that $x \leq_L y$ if and only if $f(x) \leq_{L'} f(y)$.
It is well-known that $\Lin_\fin$ is a Fraïssé class, whose limit is the linearly ordered set of rationals $\Q$.
(We recall Fraïssé theory in Section~\ref{Sec:Fraisse}.)

We shall consider ultrametric spaces with linear orders as sets of distances.
First, let $\Lin^0$ denote the category of linear orders $\seq{L, \leq_L, 0_L}$ with the least element $0_L$ and of embeddings of linear orders preserving the least element.
As usual, the triple $\seq{L, \leq_L, 0_L}$ is often denoted just by $L$, and $\leq_L$ and $0_L$ are often denoted just by $\leq$ and $0$, respectively.
Clearly, $\Lin^0$ is a category equivalent to $\Lin$.

We propose to consider ultrametric spaces as two-sorted structures of the form $\seq{X, d_X, D_X}$,
where $X$ is a set of points, $D_X$ is a set of distances, i.e. an $\Lin^0$-object, and $\map{d_X}{X \times X}{D_X}$ is an ultrametric.
To be more specific, $d = d_X$ is symmetric, $d(x, x') = 0$ if and only if $x = x'$, and $d$ satisfies the ultrametric triangle inequality 
\[
    d(x, x'') \leq \max\set{d(x,x'), d(x',x'')} \qquad \text{for $x,x',x'' \in X$.}
\]
Again, we shall denote $d_X$ by $d$ and $\seq{X, d_X, D_X}$ by $\seq{X, D_X}$ or just $X$ whenever convenient.
We shall refer to the two sorts as the \emph{point sort} and the \emph{distance sort}, respectively, when convenient.

A \emph{distance-carrying embedding} $f\maps \seq{X, d_X, D_X} \to \seq{Y, d_Y, D_Y}$ (\emph{dc-embedding} for short) is a one-to-one map $f\maps X \to Y$ together with an $\Lin^0$-map $D_f\maps D_X \to D_Y$ such that
\[
    d_Y(f(x), f(x')) = D_f(d_X(x, x')) \qquad \text{for } x, x' \in X.
\]
Note that an \emph{isometric embedding} corresponds to a dc-embedding $f\maps X \to Y$ such that $D_f$ is the inclusion $D_X \leq D_Y$, i.e. we have $d_Y(f(x), f(x')) = d_X(x, x')$.

Let $\Ult$ denote the category of ultrametric spaces and distance-carrying embeddings.
Clearly the assignment $(f\maps X \to Y) \mapsto (D_f\maps D_X \to D_Y)$ yields a functor $D\maps \Ult \to \Lin^0$ with the restriction $D_\fin \maps \Ult_\fin \to \Lin^0_\fin$ (which we may denote just by $D$).
Note that a $\Ult_\fin$-object is an ultrametric space $X$ such that both the set of points $X$ and the set of distances $D_X$ are finite.

Let $\UltIso^E \subs \Ult$ for $E \in \Lin^0$ denote the subcategory of all ultrametric spaces $X$ with $D_X \leq E$ and all isometric embeddings.
We say that an ultrametric space $X$ is \emph{precise} if $d_X\maps X \times X \to D_X$ is surjective, i.e. every distance is achieved.
By $\UltIso^E_{\rm prec} \subs \UltIso^E$
  we mean the full subcategory consisting of precise spaces. 
In this view, the category of classical ultrametric spaces is just $\UltIso^{\Rpos}_{\rm prec}$.

Since every ultrametric space $X$ can be viewed both as an $\Ult$-object and an $\UltIso^{D_X}$-object, we make a convention that $\Aut(X)$ denotes the dc-automorphism group, while $\Iso(X)$ denotes the isometry group.

A {\it precise partial dc-automorphism} $p$ of $A$ is a  partial automorphism such that all $\dom(p), \rng(p) \leq A$ are precise.
Then for every $r \in \dom(D_p)$ we have $D_p(r) = D_p(d(x, y)) = d(p(x), p(y))$ for some $x, y \in A$, i.e.  we can ignore the distance part of $D_p$.

\medskip

By the next lemma, we can often consider only precise ultrametric spaces without loss of generality as they form a cofinal subfamily in $\Ult_\fin$.

\begin{lm}\label{precise}
For any space $X \in \Ult_\fin$ there is a precise space $Y \in \Ult_\fin$ with $X \leq Y$ and $D_X = D_Y$.
\end{lm}
\begin{proof}
List all positive elements of $E = D_X$ as $r_1<\cdots<r_m$ and pick 
$m+1$ elements $a_0,a_1\ldots, a_m$ that are not in $X$.
We let $Y = X \cup \{a_0,a_1\ldots, a_m\}$ and define $d \supseteq d_X$.
Let for $j>i$, $d(a_i,a_j)=r_j$, and $d(a_i,a_i)=0$.
For $b\in X$ and $i=0,1,\ldots, m$, we let $d(b,a_i)=r_m$.
Then $\seq{Y, d, E}$ is as required.
\end{proof}

Ultrametric spaces such that  distance functions have values in  a partially ordered set with a smallest element were studied by Priess-Crampe and Ribenboim in a series of articles, e.g. \cite{PCR1}, \cite{PCR2}, \cite{PCRh}.
Note that what we call a dc-embedding appears also in \cite[p.~57]{PCR1} and is called there an isometry.
Isometric homogeneity formulated using levelwise regular branching in the tree of precise balls is considered in \cite{PCRh}.

\subsubsection{Tree representation of ultrametric spaces by precise balls}\label{sec:treerep}

It is well known that ultrametric spaces correspond to trees of balls. We will make it precise in our setup and we will introduce notation used later.

Let us fix an ultrametric space $X$.
By a \emph{ball} in $X$, we mean any subset $B \subseteq X$ that is an \emph{open ball}, i.e. a set of the form $B_r(a) = \set{x \in X: d(x, a) < r}$ for some $a \in X$ and $r \in D_X \setminus \set{0}$, or a \emph{closed ball}, i.e. a set of the form $\bar{B}_r(a) = \set{x \in X: d(x, a) \leq r}$ for some $a \in X$ and $r \in D_X$.
Note that neither the center nor the radius is unique for a general ball (in fact, every point of a ball works as its center in an ultrametric space).
A \emph{precise ball} is a closed ball of the form $\bar{B}(a, b) = \bar{B}_{d(a, b)}(a) = \set{x \in X: d(a, x) \leq d(a, b)}$, or equivalently a ball $B$ that attains its \emph{diameter} $\diam(B) = \max\set{d(x, y): x, y \in B}$. 
In particular, every singleton $\set{a}$ is a precise ball of diameter $0$.

Note that every two balls $B$, $B'$ are either disjoint, or one is contained in the other.
In the former case, we put $d(B, B') = d(x, x')$ for any $x \in B$ and $x' \in B'$, which does not depend on the choice of the points.
In the latter case, we put $d(B, B') = 0$, so in both cases we have $d(B, B') = \min\set{d(x, x'): x \in B, x' \in B'}$.

\bigskip

By a \emph{tree} we mean a partially ordered set $T$ such that for every $t \in T$ the \emph{down-set} $(\leftarrow, t] = \set{x: x \leq t}$ is linearly ordered.
This differs from the classical definition in set theory, where one requires down intervals to be well-ordered.
$T$ is a \emph{meet tree} if additionally for every $t, t' \in T$ the meet $t \meet t' = \inf\set{t, t'}$ exists.
Let $\Tree$ denote the category of meet trees and meet-preserving embeddings.
For every $t \in T$ we have the equivalence relation $\approx_t$ defined by $s \approx_t s'$ if $s \meet s' > t$ that partitions the up-set $(t, \to )=\{x\colon x>t\}$ into equivalence classes, which we call \emph{branching classes}.
To see the transitivity of $s \approx_t s' \approx_t s''$ note that $s \meet s'' \geq \min\set{s \meet s', s' \meet s''} > t$ since $(\from, s']$ is linearly ordered.
If $T$ is finite, then branching classes have immediate successors of $t$ as smallest elements.

A \emph{leveled meet tree} is a triple $\seq{T, \lev_T, \Lev(T)}$ such that $T$ is a meet tree, $\Lev(T)$ is a linear order, and $\lev_T\maps T \to \Lev(T)$, often denoted just by $\ell$, reverses the strict order, i.e. $t < t'$ implies $\lev_T(t) > \lev_T(t')$, i.e. every restriction $\lev_T\maps (\leftarrow, t] \to \Lev(T)$ is a reverse-embedding of linear orders.
Similarly to dc-embeddings of ultrametric spaces, a \emph{level-carrying $\Tree$-embedding} is a $\Tree$-map $f\maps S \to T$ together with an $\Lin$-map $\Lev(f)\maps \Lev(S) \to \Lev(T)$ such that $\lev_T(f(s)) = \Lev(f)(\lev_S(s))$ for every $s \in S$.
Let $\LevTree$ be the category of leveled meet trees and level-carrying $\Tree$-embeddings.
Clearly, we have a forgetful functor $\Lev\maps \LevTree \to \Lin$.
Let $\LevTree^0$ be the subcategory of $\LevTree$ of leveled meet trees $T$ such that $\Lev(T)$ has the smallest element $0$ and such that for every $t \in T$ there are leaves $x, y \in \ell_T^{-1}(0)$ such that $t = x \meet y$.
As morphisms of $\LevTree^0$ we take $\LevTree$-morphisms $f$ such that $\Lev(f)$ preserves $0$.
Then $\Lev$ yields a functor $\LevTree^0 \to \Lin^0$.

A connection between $\R$-valued ultrametric spaces and certain types of trees was exploited by Gao and Shao~\cite{GS} in connection with studying the isometry groups, and later by Camerlo, Marcone, and Motto Ros~\cite{CaMarRos} while describing the connection functorially.
Now we give an explicit equivalence of our categories of spaces and trees.

\begin{observation}
   $\Ult$ and $\LevTree^0$ are equivalent categories.
    For every ultrametric space $X$ we consider the meet tree $T = \seq{\Beebar_X, \sups}$ of all precise balls, and we turn it into a leveled tree by putting $\Lev(T) = D_X$ and $\ell_T = \diam_X$, where $\diam_X\maps \Beebar_X \to D_X$ is the diameter operation.
    For every dc-embedding $f\maps X \to Y$ let $\Beebar_f\maps \Beebar_X \to \Beebar_Y$ be the map $\bar{B}(x, x') \mapsto \bar{B}(f(x), f(x'))$ together with the level-carrying part $\Lev(\Beebar_f) = D_f$.
    Then $\Beebar\maps \Ult \to \LevTree^0$ is a functor.

    Now for every leveled meet tree $T$ in $\LevTree^0$ let $\Max(T) = \seq{\lev\fiber{0}, d, \Lev(T)}$ be the ultrametric space of maximal elements of $T$ where $d(x, y) = \lev(x \meet y)$, and for every $\LevTree^0$-map $f\maps S \to T$ let $\Max(f)\maps \Max(S) \to \Max(T)$ be the restriction $f\restr{\ell\fiber{0}}$ together with the distance-carrying part $D_{\Max(f)} = \Lev(f)$.
    Then $\Max\maps \LevTree^0 \to \Ult$ is a functor.

    Together, $\Beebar$ and $\Max$ realize an equivalence between the categories $\Ult$ and $\LevTree^0$ that restricts to an equivalence between $\Ult_\fin$ and $\LevTree^0_\fin$.
    Moreover, we have ${\Lev} \cmp \Beebar = D$ and $D \cmp \Max = \Lev$, i.e. the functors do nothing on the distance/level set.
\end{observation}

\begin{proof} 
    There are many thing to verify, most of which are very easy and left for the reader. We stress the following.
    \begin{itemize}
        \item $T = \seq{\Beebar_X, \sups}$ indeed has meets. If two balls $B, B' \in \Beebar_X$ are comparable, the meet is the larger one, and if they are disjoint, $B \meet B' = \bar{B}(x, x')$ for any $x \in B$ and $x' \in B'$.
        
        \item We have $\bar{B}(x, x') \subs \bar{B}(y, y')$ if and only if $\bar{B}(f(x), f(x')) \subs \bar{B}(f(y), f(y'))$.
        Let $r = d(x, x')$ and $s = d(y, y')$, so $\bar{B}(x, x') = \bar{B}_r(x)$ and $\bar{B}(y, y') = \bar{B}_s(y)$.
        Then $\bar{B}(x, x') \subs \bar{B}(y, y')$ if and only if $d(x, y) \leq s$ and $r \leq s$, which is equivalent to $d(f(x), f(y)) \leq D_f(s)$ and $D_f(r) \leq D_f(s)$ as $f$ is a dc-embedding.
        
        \item The previous point implies that $\Beebar_f$ is well-defined and that it preserves meets.

        \item For every $X \in \Ult$ we have that $\Max(\Beebar_X)$ is a copy of $X$ where every point $x$ is replaced with the singleton precise ball $\set{x}$, which yields a natural isomorphism.

        \item Similarly, for every $T \in \LevTree^0$, $\Beebar_{\Max(T)}$ is a copy of $T$ where every node $t$ is replaced with the family $[t, \to) \cap \ell^{-1}(0)$ of all leaves above $t$.
        This faithfully reconstructs the tree since we assumed that for every $t \in T$ there are $x, y \in \ell^{-1}(0)$ with $t = x \meet y$.
        \qedhere
    \end{itemize}
\end{proof}

\subsubsection{Tree representation of ultrametric spaces by adjacent balls}\label{sec:adjballs}

Replacing precise balls with so-called adjacent balls we arrive at the notion of a tree of adjacent balls.
In Proposition~\ref{prop: beta} and Section~\ref{sec:cloU} we will use those trees to represent automorphism groups of Cauchy completions of ultrametric spaces, and we will show that there is a generic leveled adjacency tree.

Let $X$ be an ultrametric space, let $B, B' \subs X$ be disjoint balls, and let $r = d(B, B')$.
Let us put $E_{B, B'} = B_r(x)$ and $E_{B', B} = B_r(x')$ for any $x \in B$ and $x' \in B'$.
Note that the open balls $E_{B, B'}$ and $E_{B', B}$ are the largest balls that are disjoint and contain $B$ and $B'$, respectively.
We say that $B, B'$ are \emph{adjacent}, and we write $B \sim B'$, if $B = E_{B, B'}$ and $B' = E_{B', B}$.
Note that the $\seq{E_{B, B'}, E_{B', B}}$ is the unique pair of adjacent balls containing $B$ and $B'$.
We call it the \emph{adjacent envelope} of $\seq{B, B'}$.
We shall call a single (necessarily open) ball $B$ \emph{adjacent} if it is adjacent to some other ball $B'$.

For every $x \neq y \in X$, similarly to $\bar{B}(x, y)$, we put $B(x, y) = B_r(x)$ where $r = d(x, y)$.
Let $B, B' \subs X$ be disjoint balls as before, and let $x \in B$ and $y \in B'$ be any points.
We have that $\seq{B(x, y), B(y, x)}$ is the adjacency envelope of $\seq{B, B'}$.
In particular, the ball $B(x, y)$ is adjacent, and every adjacent ball is of this form.
Let $\Bee_X = \set{B(x, y): x \neq y \in X}$ denote the family of all adjacent balls of $X$.

For every adjacent ball $B \subs X$ we define its \emph{radius} $r(B)$ as $d(x, y)$ for any $x, y$ such that $B = B(x, y)$, or equivalently $r(B) = \max\set{r \in D_X: B = B_r(x)} > 0$ where $x$ is any point of $B$.
Clearly, $B \sim B'$ implies $r(B) = r(B')$.
Note that in a finite space a ball can be simultaneously precise and adjacent, but the diameter and radius are different.

\begin{observation}
    For any disjoint balls $B$, $B'$, $B''$, if $B \sim B'$ and $B' \sim B''$, then $B \sim B''$, so the adjacency relation is the anti-reflexive modification of an equivalence relation on the family of all balls of $X$.
    Adjacent balls form nondegenerate equivalence classes, while non-adjacent balls form singleton equivalence classes.
    Moreover, nondegenerate precise balls (of diameter $r$) are exactly unions of nondegenerate adjacency equivalence classes (of adjacent balls of radius $r$), which correspond to the branching classes.
\end{observation}

\begin{uwgi}
In a finite space $X$, every open ball $B \subsetneq X$ is adjacent. But in infinite spaces there may exist nontrivial open balls that are not adjacent, e.g. a ball $B_r(x)$ such that for every $r' > r$ there is $y$ with $r < d(x, y) < r'$, but there is no $y$ with $d(x, y) = r$.
\end{uwgi}

Motivated by the the concrete adjacency structure $\seq{\Bee_X, \sups, \sim}$ of balls, we define the {\it adjacency relation} on any tree $T$ by $$t \sim t' \iff (\from, t) = (\from, t') \text{ and }t \neq t'.$$
By an \emph{adjacency tree} we mean a tree $T$ such that:
\begin{itemize}
    \item For every $t \in T$ there is $t' \neq t$ with $t' \sim t$.
    \item For any incomparable $t, t' \in T$ there are $e \leq t$ and $e' \leq t'$ in $T$ such that $e \sim e'$.
    We call the pair $\seq{e, e'}$ the \emph{adjacency envelope} of $\seq{t, t'}$.
\end{itemize}
Note that for every incomparable pair $\seq{t, t'}$, its adjacency envelope is unique.
Indeed, if $\seq{e, e'}$ and $\seq{f, f'}$ are two adjacency envelopes, then we have $e, f \leq t$, and so they are comparable.
If $e < f \sim f'$, then $e < f'$, which is a contradiction with $e', f' \leq t$ being comparable.
Similarly, any of $e > f$, $e' < f'$, and $e' > f'$ lead to a contradiction, so $\seq{e, e'} = \seq{f, f'}$.

Let $\ATree$ denote the category of all adjacency trees and all adjacency preserving and adjacency reflecting tree embeddings, i.e. maps $f\maps S \to T$ such that for every $s, s' \in S$ we have $s \leq s'$ if and only if $f(s) \leq f(s')$, and $s \sim s'$ if and only if $f(s) \sim f(s')$.

\begin{lm} \label{thm:adjacency_preserving}
    For a map $f\maps S \to T$ between adjacency trees to be an $\ATree$-map it is enough that it preserves order and adjacency.
\end{lm}
\begin{proof}
    The map $f$ reflects adjacency:
    if $f(s) \sim f(s')$, then $s, s'$ are incomparable, and so the adjacency envelope $s \geq e \sim e' \leq s'$ exists, and we have $f(s) \geq f(e) \sim f(e') \leq f(s')$.
    Hence, $f(s) = f(e)$ and $f(s') = f(e')$, and we have $s = e \sim e' = s'$.

    The map $f$ reflects order:
    suppose $f(s) \leq f(s')$.
    We consider several cases.
    If $s \leq s'$, we are done.
    If $s, s'$ are incomparable, we consider the adjacency envelope $s \geq e \sim e' \leq s'$, and we have $f(s) \geq f(e) \sim f(e') \leq f(s')$, and so $f(s), f(s')$ are incomparable, which is a contradiction.
    If $s > s'$, then $f(s) = f(s')$, but there is $t \sim s$, and so $t > s'$, and we have $f(s') \leq f(t) \sim f(s) = f(s')$, which is a contradiction.
\end{proof}

As with leveled meet trees, we consider \emph{leveled adjacency trees} $\seq{T, \ell_T, \Lev(T)}$: we assume that $T$ is an adjacency tree, $\Lev(T) \in \Lin^0$ is a linear order, and $\lev_T\maps T \to \Lev(T)$, often denoted just by $\lev$, is a map such that for every $t, t' \in T$ we have:
\[
    t < t' \implies \lev(t) > \lev(t'), \quad\qquad
    t \sim t' \implies \lev(t) = \lev(t'), \quad\qquad
    \lev\fiber{0} = \emptyset.
\]
We define the category $\Ball$ such that objects are leveled adjacency trees, and a $\Ball$-map $f\maps \seq{S, \lev_S, \Lev(S)} \to \seq{T, \lev_T, \Lev(T)}$ is an $\ATree$-map $f\maps S \to T$ together with an $\Lin^0$-map $\Lev(f)\maps \Lev(S) \to \Lev(T)$ such that $f$ is \emph{level-carrying}, i.e. $\lev_T(f(s)) = \Lev(f)(\lev_S(s))$ for every $s \in S$.
The projection to the level sort induces a functor $\Lev\maps \Ball \to \Lin^0$.

Let $X \in \Ult$.
By the earlier observations $T = \seq{\Bee_X, \sups,\sim}$ is an adjacency tree.
If we put $\Lev(T) = D_X$ and $\lev_T(B) = r(B)$, we obtain a leveled adjacency tree $\seq{T, \lev_T, \Lev(T)} \in \Ball$.
In this way we can view every $\Bee_X$ as a $\Ball$-object.

For every dc-embedding $f\maps X \to Y$ we consider a $\Ball$-map $\Bee_f\maps \Bee_X \to \Bee_Y$ defined by
$\Bee_f(B(x, y)) = B(f(x), f(y))$,
or equivalently by $\Bee_f(B_r(x)) = B_{D_f(r)}(f(x))$, where $r=d(x,y)$.

\begin{prop}\label{prop:thebee}
    $\Bee_f\maps \Bee_X \to \Bee_Y$ is a $\Ball$-map with $\Lev(\Bee_f) = D_f$, and so we have a functor $\Bee\maps \Ult \to \Ball$ with ${\Lev} \cmp \Bee = D$.
\end{prop}

\begin{proof} \hfill
\begin{itemize}
\item $\Bee_f$ is well-defined since it maps an adjacent ball $B$ of radius $r$ to $B_q(f(x))$ where $q = D_f(r)$ and $x \in B$. This does not depend on the choice of $x$:
for $x, x' \in B$ we have $d(x, x') < r$, and so $d(f(x), f(x')) < q$, and so $B_q(f(x)) = B_q(f(x'))$.

\item $\Bee_f$ preserves order: if $x \in B \subs B'$, then $r(B) \leq r(B')$, and so 
\[
\Bee_f(B) = B_{D_f(r(B))}(f(x)) \subs B_{D_f(r(B'))}(f(x)) = \Bee_f(B').
\]

\item $\Bee_f$ preserves adjacency since every pair $B \sim B'$ is of the form $B(x, y) \sim B(y, x)$ and so $\Bee_f(B) = B(f(x), f(y)) \sim B(f(y), f(x)) = \Bee_f(B')$.

\item $\Bee_f$ is an $\ATree$-map by Lemma~\ref{thm:adjacency_preserving}.

\item $\Bee_f$ is level-carrying with $\Lev(\Bee_f) = D_f$ since for every $x, y \in X$ we have 
\begin{align*}
    \lev(\Bee_f(B(x, y))) = r(B(f(x), f(y))) &= d(f(x), f(y)) \\ &= D_f(d(x, y)) = D_f(r(B(x, y))).
    \qedhere
\end{align*}
\end{itemize}
\end{proof}

\begin{lm} \label{thm:B_essentially_wide}
    $\Bee_\fin$ is essentially surjective on objects, i.e. for every $T \in \Ball_\fin$ there is $X \in \Ult_\fin$ such that $\Bee_X \cong T$.
\end{lm}
\begin{proof}
    Note that $T$ is a leveled adjacency tree $T=\seq{T, \lev_T, \Lev(T)}$.
    Let $X$ be the set of all maximal elements of $T$ (i.e. the leaves), and let $D_X = \Lev(T)$.
    For $t \neq t' \in X$ we put $d_X(t, t') = \lev_T(e)$ where $\seq{e, e'}$ is the adjacency envelope of $\seq{t, t'}$ in $T$.

    Clearly $d = d_X$ is symmetric, and $d(x, x') = 0$ if and only if $x = x'$.
    We put $\Bee_X = \set{B_X(x, y): x \neq y}$, where $B_X(x, y) = \set{z: d(x, z) < d(x, y)}$, even though we have not yet proved the ultrametric triangle inequality for $d_X$. That will follow once we show that $\Bee_X$ is a tree.
    First we show that $\Phi$ mapping every $t \in T$ to $B_t = \set{x \in X: x \geq t}$ is an isomorphism of posets $T \to \Bee_X$.
    
    \medskip
    \noindent
    {\bf Claim.}
    We have $B_X(x, x') = B_t$ wherever $x \geq t \sim t' \leq x'$ for some $x, x' \in X$ and $t, t' \in T$.
    \medskip

    \noindent
    {\it Proof of Claim.}
    Whenever $y \in X \setminus \set{x}$, we let $\seq{e_{x, y}, e_{y, x}}$ denote the adjacency envelope.
    If $y \in B_t$, we have $x, y \geq t$, and so either $y = x$, or $e_{x, y}, e_{y, x} > t$ and $d(x, y) = \lev(e_{x, y}) < \lev(t) = d(x, x')$.
    In both cases $y \in B_X(x,x')$.

    If $y \in B_X(x, x')$, then either $y = x$, or $\ell(e_{y, x}) = d(y, x) < d(y, x') = \ell(e_{y, x'})$.
    In the second case we have $e_{x, y} \sim e_{y, x} > e_{y, x'}$ and so $x > e_{y, x'} \sim e_{x', y} \leq x'$.
    Hence, $\seq{e_{y, x'}, e_{x', y}} = \seq{t, t'}$, and $y > t$.
    In both cases, $y \geq t$, i.e. $y \in B_t$.
    \hfill $\qed_{\text{Claim}}$
    
    \medskip

    Given $t \in T$, let $t' \sim t$ and let $x \geq t$ and $x' \geq t'$ be any maximal elements of $T$, i.e.
    $x, x' \in X$.
    By the Claim, $B_t = B_X(x, x') \in \Bee_X$.
    On the other hand, given $B \in \Bee_X$, we know that $B$ is of the form $B_X(x, x')$ for some $x, x' \in X$.
    Let $\seq{t, t'}$ be the adjacency envelope of $\seq{x, x'}$ in $T$.
    Then by the Claim, we have $B = B_t$, and $\Phi$ is surjective.

    Now let $t, t' \in T$.
    Clearly, if $t \leq t'$, then $B_t \supseteq B_{t'}$.
    Also, if $t, t'$ are incomparable, then $B_t \cap B_{t'} = \emptyset$, and so they are $\subs$-incomparable.
    Moreover, if $t < t'$, then $B_t \supsetneq B_{t'}$ since there is $t'' \sim t'$ and $B_t \setminus B_{t'} \supseteq B_{t''}$.
    This together with surjectivity shows that $\Phi\maps T \to \Bee_X$ is an isomorphism of posets.
    Hence, $\Bee_X$ is a tree, and both $\Phi$ and $\Phi^{-1}$ preserve the adjacency relation.
    
    Since $\Bee_X$ is a tree, $d$ is an ultrametric:
    If we had $d(x, y) > d(x, z), d(y, z)$, then $B_X(x, y)$ and $B_X(y, x)$ would be incomparable balls containing $z$, giving a contradiction.
    Finally, for $t \in T$ and $x, x' \in X$ as in the Claim, we have $r(B_t) = r(B_X(x, x')) = d(x, x') = \lev_T(t)$, and so $\Phi$ is level-preserving.
    Altogether, $\Phi\maps T \to \Bee_X$ is a $\Ball$-isomorphism.
\end{proof}

\begin{observation} \label{thm:unique_ball_pair}
    Note that for every $X \in \Ult$ and $x \neq y \in X$ we have that $\seq{B(x, y), B(y, x)}$ is the unique pair $\seq{B, B'}$ of adjacent balls in $\Bee_X$ such that $x \in B$ and $y \in B'$.
\end{observation}

Let us for now denote by $\Ult^+_\fin$ denote the full subcategory of $\Ult_\fin$ consisting of spaces with nonempty point sort.

\begin{lm} \label{thm:B_star_surjective}
    For every $X \in \Ult_\fin$  and $Y \in \Ult$ (with the single exception of $X = \seq{\set{*}, D_X}$ and $Y = \seq{\emptyset, D_Y}$) and every $\Ball$-map $g\maps \Bee_X \to \Bee_Y$ there is a $\Ult$-map $f\maps X \to Y$ such that $\Bee_f = g$.
    In particular, $\Bee_\fin\maps \Ult^+_\fin \to \Ball_\fin$ is full.
\end{lm}
\begin{proof}
    If $X$ has at most one point, then $\Bee_X = \emptyset$, and so $g$ is the empty map, and any map $f\maps X \to Y$ is a dc-embedding with $\Bee_f = g$.
    In the following we suppose that $X$ has at least two points.
    
    For every $x \in X$ let $r_x = \min\set{d(x, y): y \in X \setminus \set{x}}$, so that $B_{r_x}(x)$ is the smallest adjacent ball in $X$ containing $x$.
    We let $f(x)$ be any point of $g(B_{r_x}(x))$, and of course we put $D_f = \Lev(g)$.
    We need to check that $f$ is a dc-embedding and that $\Bee_f = g$.

    Let $x \neq y \in X$. We show that $g(B(x, y)) = B(f(x), f(y))$.
    In $\Bee_X$ we have $B_{r_x}(x) \subs B(x, y) \sim B(y, x) \supseteq B_{r_y}(y)$, and therefore 
    \[
        f(x) \in g(B_{r_x}(x)) \subs g(B(x, y)) \sim g(B(y, x)) \supseteq g(B_{r_y}(y)) \ni f(y).
    \]
    It follows from Observation~\ref{thm:unique_ball_pair} that 
    \[
        \seq{g(B(x, y)), g(B(y, x))} = \seq{B(f(x), f(y)), B(f(y), f(x))}.
    \]
    In particular, $g(B(x, y)) = B(f(x), f(y))$.
    Hence, we have
    \begin{align*}
        D_f(d(x, y)) = \Lev(g)(r(B(x, y))) &= r(g(B(x, y))) \\ &= r(B(f(x), f(y))) = d(f(x), f(y)),
    \end{align*}
    and that $f$ is a dc-embedding with $D_f = \Lev(g)$.
    Finally, the fact that $g(B(x, y)) = B(f(x), f(y))$ for every $x \neq y \in X$ exactly means that $\Bee_f = g$.
\end{proof}

\begin{uwgi}
    The functor $\Bee$ (as opposed to $\Bee_\fin$) is not full since a convergent sequence with and without the limit have the same ball space, and the identity embedding between them does not lift to points.
\end{uwgi}

\begin{uwgi}
    Every $X \in \Ult_\fin$ with at least two points can be reconstructed from $\Bee_X$: the points of $X$ are exactly the elements of singleton balls in $\Bee_X$.
    Hence, the functor $\Bee_\fin\maps \Ult^+_\fin \to \Ball_\fin$ is injective on objects.
    By the above lemmas this functor is also essentially surjective on objects and full.
    Altogether, it is essentially a quotient: the categories have essentially the same objects but some morphisms in homsets are identified.
\end{uwgi}

In  Figure~\ref{fig:zoo} we summarize the categories and functors introduced above.

\begin{figure}[!ht] 
\centering
\begin{tikzpicture}[
    x = {(6em, 0em)},
    y = {(0em, 4em)},
    label/.style = {
        edge label={#1},
        font=\footnotesize,
    },
]
    \node (I) at (0, 0) {$\UltIso^E$};
    \node (U) at (1, 0) {$\Ult$};
    \node (TL) at (2, 1) {$\LevTree^0$};
    \node (T) at (3, 1) {$\Tree$};
    \node (B) at (2, -1) {$\Ball$};
    \node (A) at (3, -1) {$\ATree$};
    \node (L) at (3, 0) {$\Lin^0$};

    \graph{
        (I) ->[label=$\subs$] (U) ->[label=$D$, pos=0.6] (L),
        (U) ->[label=$\Bee$, inner sep=2pt, swap] (B) ->[label=$\Lev$, swap, inner sep=2pt] (L),
        (B) -> (A),
        (U) ->[label=$\Beebar$, inner sep=2pt, bend left=10] (TL) ->[label=$\Lev$, inner sep=2pt] (L),
        (TL) ->[label=$\Max$, inner sep=1pt, pos=0.4, bend left=10] (U),
        (TL) -> (T),
    };
\end{tikzpicture}

\caption{The introduced categories and functors between them.}
\label{fig:zoo}
\end{figure}

\subsection{Topologies on ultrametric spaces and their automorphism groups}\label{sec:topandCau}

As with classical ultrametric spaces, we have the metric topology and uniformity: a subset $U \subs X$ is open if for every $a \in U$ there is some $r \in D_X \setminus \set{0}$ such that $B_r(a) = \set{x \in X: d(x, a) < r} \subs U$, and each $r \in D_X \setminus \set{0}$ induces the partition $\set{B_r(a): a \in X}$. The family of these partitions forms a basis of a uniformity on $X$.
Hence, topological and uniform notions like a subset being dense or a space being uniformly complete are well-defined.

We note that a dc-embedding $f\maps X \to Y$ (or even an isometric embedding with $D_X \subsetneq D_Y$) may not be continuous.
The key ingredient is that $D_f$ is \emph{continuous at $0$}, i.e. there is no $\eps \in D_Y$ such that $0 < \eps < D_f\im{D_X \setminus \set{0}}$.
An ultrametric space $X$ is \emph{discrete} if for every $x \in X$ there is $\eps \in D_X \setminus \set{0}$ such that $B_\eps(x) = \set{x}$, and it is \emph{uniformly discrete} if moreover $\eps$ can be chosen independently of $x$.
\begin{lm} \label{thm:dc-continuity}
    Let $f\maps X \to Y$ be a dc-embedding of ultrametric spaces.
    \begin{enumerate}
        \item $f$ is continuous if and only if $D_f$ is continuous at $0$ or $X$ is discrete.
        \item $f$ is uniformly continuous if and only if $D_f$ is continuous at $0$ or $X$ is uniformly discrete.
    \end{enumerate}
    In particular, a dc-isomorphism is always a (uniform) homeomorphism.
\end{lm}
\begin{proof}
    Continuity means that for every $\eps \in D_Y \setminus \set{0}$ and $x \in X$ there is $\delta \in D_X \setminus \set{0}$ such that $f\im{B_\delta(x)} \subs B_\eps(f(x))$, and uniform continuity means the same with $\delta$ not depending on $x$.
    If $D_f$ is continuous at $0$, then we have $\delta \in D_X \setminus \set{0}$ with $D_f(\delta) \leq \eps$, and so $f$ is uniformly continuous.

    On the other hand, if $\eps$ witnesses that $D_f$ is not continuous at $0$ and $\delta$ witnesses that $f$ is (uniformly) continuous, then $B_\delta(x) = \set{x}$, and so $X$ is (uniformly) discrete.
\end{proof}

We say that an ultrametric space $X$ is \emph{Cauchy complete} if it is complete with respect to the induced uniformity, which means that every \emph{Cauchy chain} of balls, i.e. a chain of balls $\C$ such that for every $r \in D_X \setminus \set{0}$ there is $B \in \C$ with $B \subs B_r(x)$ for some $x$, has a non-empty intersection, which is necessarily a singleton.
It does not matter if we consider open, closed, precise, or adjacent balls: indeed a Cauchy chain $\C$ either contains a singleton $\set{x}$, which is then its intersection,
or else for any $B \in \C$ there is $B' \in \C$ such that $B \supsetneq B'$.
Then for $x \in B'$ and $y \in B \setminus B'$ we have $B \sups \bar{B}(x, y) \sups B(x, y) \sups B'$.

The \emph{Cauchy-completion} $\overline{X}$ of an ultrametric space $X$, which is the unique Cauchy complete extension having $X$ as a dense subspace and having the same distance set, can be constructed as follows.
The space $\seq{\clo{X},d_{\clo{X}},D_X}$ consists of all maximal Cauchy chains $\C$ of open balls.
The distance $d_{\clo{X}}: \clo{X} \times \clo{X} \to D_X$ is defined as $d_{\overline{X}}(\C,\C')=d(B,B')$, where $B \in \C$, $B' \in \C'$ are any disjoint balls. Note that by the ultrametric triangle inequality, this does not depend on the choice of $B$, $B'$.
The mapping $x \mapsto \{B_r(x): r \in D_X \setminus \set{0}\}$ is an isometric embedding of $X$ onto a dense subspace of its Cauchy-completion $\clo{X}$.

\begin{observation} \label{exttoCauchy}
    For every uniformly continuous dc-embedding $f\maps X \to Y$, the unique uniformly continuous map $\bar{f}\maps \clo{X} \to \clo{Y}$ is a dc-embedding with $D_{\bar{f}} = D_f$ since for every $\bar{x}, \bar{y} \in \clo{X}$ with $r = d(\bar{x}, \bar{y}) > 0$ there are $x \in B_r(\bar{x}) \cap X$ and $y \in B_r(\bar{y}) \cap X$, and so we have 
    $d(\bar{f}(\bar{x}), \bar{f}(\bar{y})) = d(f(x), f(y)) = D_f(d(x, y)) = D_f(d(\bar{x}, \bar{y}))$.
    
    It follows from the uniqueness that $\clo{g \cmp f} = \bar{g} \cmp \bar{f}$.
    In particular, we have the canonical group embedding $\alpha\maps \Aut(X) \to \Aut(\clo{X})$ mapping $f \mapsto \bar{f}$, and $\Aut(X)$ can be identified with $\set{f \in \Aut(\clo{X}): f\im{X} = X}$.
\end{observation}

On $\Aut(X)$ we consider two topologies:
\begin{itemize}
    \item The topology of pointwise convergence, viewing $X$ as a discrete two-sorted structure $X \cup D_X$, i.e. the neighborhood basis at every $f \in \Aut(X)$ consists of the sets
    \[
        N_A(f) = \set{g \in \Aut(X): g(x) = f(x)\text{ for $x \in A$, and $D_g(r) = D_f(r)$ for $r \in D_A$}},
    \]
    where $A \subs X$ is a finite subspace.
    \item The topology of pointwise convergence, viewing $X$ as a topological space with the topology induced by the ultrametric, i.e. the neighborhood basis at every $f \in \Aut(X)$ consists of the sets
    \[
        N_{A, r}(f) = \set{g \in \Aut(X): d(g(x), f(x)) < r \text{ for every $x \in A$}},
    \]
    where $A \subs X$ is a finite subspace and $r \in D_X \setminus \set{0}$.
\end{itemize}

    The first topology is a standard choice for any countable structure $M$, and turns $\Aut(M)$ into a Polish group. We will use this topology without mention in cases like $\Aut(\Qpos)$ and $\Aut(\Bee_X)$.
The second topology makes a good sense only for precise spaces $X$ as there is no way how the distance part can be controlled on unused distances.

We make a convention that $\Aut(X)$ denotes the group with the first topology, while $\AutM(X)$ denotes the group with the second topology.
Similarly we denote their respective topological subgroups $\Iso(X)$ and $\IsoM(X)$.

Clearly, the topology of $\AutM(X)$ is coarser than the topology of $\Aut(X)$.
It is a standard fact that $\Aut(X)$ is a Polish group if $X$ is countable (and for general $X$, $\Aut(X)$ is a completely uniformizable topological group).
We will observe that $\AutM(X)$ is a topological group.

\begin{observation} \label{thm:distance_neighborhood}
    Let $f \in \AutM(X)$.
    If $D_f(r) = q$,  then also $D_g(r) = q$ for every $g \in N_{\set{x, y}, q}(f)$ where  $x, y \in X$ are such that $d(x,y)=r$.
    This follows from the ultrametric triangle inequality applied to $g(x), f(x), f(y), g(y)$.
  
\end{observation}

\begin{lm} \label{thm:AutM}
    $\AutM(X)$ is a topological group for any precise ultrametric space $X$.
\end{lm}

\begin{proof}
    To prove continuity of the composition, let $f, f' \in \AutM(X)$, $A \subs X$ be finite, and $q \in D_X \setminus \set{0}$.
    For $r = D_{f'}^{-1}(q)$ there are points $x, y$ with $d(x, y) = r$.
    We put $B = f[A] \cup \set{x, y}$.
    Then $N_{B, q}(f') \cmp N_{A, r}(f) \subs N_{A, q}(f' \cmp f)$:
    for every $a \in A$, $g \in N_{A, r}(f)$, and $g' \in N_{B, q}(f')$ we have
    \begin{itemize}
        \item $d(g(a), f(a)) < r$ since $a \in A$ and $g \in N_{A, r}(f)$,
        \item $d(g'(g(a)), g'(f(a))) < q$ since $D_{g'}(r) = q$, by Observation~\ref{thm:distance_neighborhood}, and the fact that $x, y \in B$,
        \item $d(g'(f(a)), f'(f(a))) < q$ since $f(a) \in B$ and $g' \in N_{B, q}(f')$,
    \end{itemize}
    and altogether $d(g'(g(a)), f'(f(a))) < q$.

    For continuity of the inverse we similarly have 
    $N_{A, q}(f)^{-1} \subs N_{B, r}(f^{-1})$ where $f, B, r$ are given, and $q = D_f(r)$ and $A = f^{-1}[B] \cup \set{x, y}$ where $d(x, y) = r$.
    This is because $d(g^{-1}(b), f^{-1}(b)) < r$ if and only if $d(f(a), g(a)) < q$ where $f(a) = b$ (and $a = f^{-1}(b)$) and $g \in N_{A, q}(f)$ and so $D_g(r) = q$ (and $r = D_{g^{-1}}(q)$).
\end{proof}

For every ultrametric space $X$, by the \emph{canonical projection} $\pi\maps \Aut(X) \to \Aut(D_X)$ we mean the homomorphism defined by $\pi(f) = D_f$.

\begin{observation} \label{thm:canonical_projection_continuous}
    The canonical projection $\pi\maps \Aut(X) \to \Aut(D_X)$ is continuous for every space $X$ as $D_X$ is already part of the structure of $X$.
    The canonical projection $\pi\maps \AutM(X) \to \Aut(D_X)$ is continuous for every precise space $X$ by Observation~\ref{thm:distance_neighborhood}.

    It follows that $\Iso(X)$ and $\IsoM(X)$ are closed normal subgroups of $\Aut(X)$ and $\AutM(X)$, respectively.
\end{observation}

Recall that for every ultrametric space $X$ we denote by $\Bee_X$ the family of all adjacent balls of $X$ and that this gives a functor $\Bee\maps \Ult \to \Ball$ with $\ell \cmp \Bee = D$.

For every ultrametric space $X$ we consider the homomorphism $\beta\maps \Aut(X) \to \Aut(\Bee_X)$ defined by $\beta(f) = \Bee_f$, with $\ell_{\beta_f} = D_f$.
Clearly, $\beta$ is one-to-one: if $f(x) \neq x$, then $\Bee_f(B) \neq B$ for $B = B(x, f(x))$.

\begin{lm} \label{thm:dense_embedding}
    If $f\maps X \to Y$ is a dc-embedding into a precise space such that $f[X] \subs Y$ is dense, then $D_f\maps D_X \to D_Y$ is an $\Lin^0$-isomorphism, and $\Bee_f\maps \Bee_X \to \Bee_Y$ is a $\Ball$-isomorphism.    
\end{lm}
\begin{proof}
    For $r \in D_Y$ there are points $y, y' \in Y$ with $d(y, y') = r$, and, by density, there are points $x, x' \in X$ with $d(f(x), y), d(f(x'), y') < r$.
    Hence, $d(f(x), f(x')) = r$, and so $D_f$ is surjective.

    Similarly, for every $r\in D_Y\setminus\{0\}$ and $B(y, y') \in \Bee_Y$ with $d(y,y')=r$ there are $x, x' \in X$ with $d(f(x), y), d(f(x'), y')< r$.
    Hence, $B(y, y') = B(f(x), f(x')) = \Bee_f(B(x, x'))$, and so $\Bee_f$ is surjective.
\end{proof}

We say that a point $x \in X$ in an ultrametric space is \emph{precisely isolated} if there is $y \in X$ such that $B(x, y) = \set{x}$.
\begin{prop}\label{prop: beta}
    Let $X$ be a precise ultrametric space such that every isolated point is precisely isolated.
    \begin{enumerate}
        \item The map $\beta\maps \AutM(X) \to \Aut(\Bee_X)$ is an embedding of topological groups.
        \item If $X$ is separable, then $\Aut(\Bee_X)$ is Polish, and so $\AutM(X)$ is separable metrizable.
        \item If $X$ is separable and complete, then $\beta\maps \AutM(X) \to \Aut(\Bee_X)$ is an isomorphism of Polish groups.
        \item If $X$ is separable, then $\bar{\beta}\maps \AutM(\clo{X}) \to \Aut(\Bee_X)$ is an isomorphism of Polish groups.
        Here $\bar{\beta}$ denotes $\Bee_i^{-1} \cmp \beta_{\clo{X}}\maps \AutM(\clo{X}) \to \Aut(\Bee_{\clo{X}}) \to \Aut(\Bee_X)$ where $i\maps X \to \clo{X}$ is the inclusion.
    \end{enumerate}
\end{prop}
\begin{proof}
    To prove~(1) we have to show that $\beta$ and its inverse are continuous. Let $f \in \AutM(X)$ and $B \in \Bee_X$.
    For $q = D_f(r(B))$ and any $x \in B$ we have $\set{g \in \AutM(X): g[B] = f[B]} = N_{\set{x}, q}(f)$, and so $\beta$ is continuous.
    On the other hand, if $x \in X$ and $q \in D_X \setminus \set{0}$, then there is $0 < q' \leq q$ such that $B_{q'}(f(x)) \in \Bee_X$ since every isolated point is precisely isolated.
    For $r = D_f^{-1}(q')$ we have $B_r(x) = f\preim{B_{q'}(f(x))} \in \Bee_X$, and so if $g[B_r(x)] = f[B_r(x)]$, then $g \in N_{\set{x}, q}(f)$.
    Hence, $\beta$ is an embedding.
    
    Claim~(2): as $X$ is separable, there is countable dense subset $A \subs X$, and so $\Bee_X = \set{B(x, y): x \neq y \in A}$ is countable.
    As $X$ is precise, $\Lev(\Bee_X) = D_X$ is countable as well.  Therefore $\Aut(\Bee_X)$ as the automorphism group of a countable structure is Polish.

    To prove (3) it is enough to show that $\beta$ is surjective.
    For every $x \in X$ let $\C_x$ be the maximal chain $\set{B(x, y): y \in X \setminus \set{x}} \subs \Bee_X$.
    Since every isolated point of $X$ is precisely isolated, either $\C_x$ has a minimal element, or $\inf\set{r(B): B \in \C_x} = 0$.
    Since $X$ is complete, $x \mapsto \C_x$ is a bijective correspondence between points and such chains.
    Hence, for $g \in \Aut(\Bee_X)$ and $x \in X$ we have $\set{g(B): B \in \C_x} = \C_{f(x)}$ for a unique point $f(x)$.
    This defines a bijection $f\maps X \to X$ such that $g(B) = f[B]$ for every $B \in \Bee_X$ since we have
    \[
        x \in B \iff B \in \C_x \iff g(B) \in \C_{f(x)} \iff f(x) \in g(B).
    \]
    It remains to show that $f$ is a dc-automorphism with $D_f = \ell_g$.
    For every $x \neq y \in X$ we have $d(x, y) = r(B(x, y))$, and $B(x, y) = \max_\subs(\C_x \setminus \C_y)$, and 
    \[
        \max(\C_{f(x)} \setminus \C_{f(y)}) = \max\set{g(B): B \in \C_x \setminus \C_y} = g(\max(\C_x \setminus \C_y)).
    \]
    Hence, $B(f(x), f(y)) = g(B(x, y))$ and $d(f(x), f(y)) = \ell_g(d(x, y))$.
    
    Claim~(4):
    by the first part of Lemma~\ref{thm:dense_embedding} we have that $D_{\clo{X}} = D_X$, so $\clo{X}$ is precise.
    By density, every isolated point of $\clo{X}$ is in $X$ and is precisely isolated.
    Hence, the claim follows from (3) applied to $\clo{X}$ and from the second part of Lemma~\ref{thm:dense_embedding}.
\end{proof}

\section{Generic ultrametric spaces}
\label{Sec:Fraisse}

We shall recall notions from Fraïssé theory and adapt them to our context.
Classical Fraïssé theory of countable one-sorted structures is well-known, see Fraïssé~\cite{Fraisse} and Hodges~\cite[Section~7.1]{Hodges}.
We intend to use the theory for categories of structures possibly with partial automorphisms.
We stress the category-theoretic aspects so that we can conveniently relate different classes by forgetful functors and use abstract principles to transfer Fraïssé-theoretic properties between them.
Category-theoretic Fraïssé theory was originally developed by Droste and Göbel~\cite{DrG} and Kubiś~\cite{Kub40}.

Let $\K$ be a category of structures. In particular, $\K$-objects are (many-sorted) structures, and $\K$-maps are all embeddings.
Fraïssé theory relates properties of the class $\K_\fin$ of finite structures and of countable structures that are colimits of countable direct sequences from $\K_\fin$, i.e. up to isomorphism they are unions of countable increasing chains.
We denote this class by $\sig\K_\fin$.
We assume that $\K$ is \emph{$\sigma$-complete} in the sense that it is closed under countable unions and their isomorphic copies (i.e. under colimits of countable sequences taken in the ambient category of all structures of a given many-sorted signature).

A sequence $(A_n, f_n)$ in $\K$ consists of $\K$-objects $A_n$ and $\K$-maps $f_n\maps A_n \to A_{n + 1}$, $n \in \N$.
We also introduce notation for compositions of the bonding maps: $f_m^n\maps A_m \to A_n$ for $m \leq n \in \N$.
The colimit is denoted by $(A_\infty, f^\infty_n)$, i.e. it consists of a $\K$-object $A_\infty$ and of $\K$-maps $f_n^\infty\maps A_n \to A_\infty$ forming a cone.

Below we summarize the relevant properties of $\K_\fin$ such as variants of the amalgamation, properties of objects from $\sig\K_\fin$ such as homogeneity and extension property, and we formulate a general Fraïssé theorem in our context.

We define the following properties for $\K_\fin$.
\begin{itemize}
    \item[(HP)] The {\it hereditary property}: for every $A\in \mathcal{K}_\fin$ and $B\leq A$ we have $B\in \mathcal{K}_\fin$.
    \item[(JEP)] The {\it joint embedding property}: for any $A,B\in\mathcal{K}_\fin$ there is $C\in \mathcal{K}_\fin$, which embeds both $A$ and $B$.
    \item[(AP)] The {\it amalgamation property}: for every $A\in \mathcal{K}_\fin$, $\alpha_1\colon A\to B$ and $\alpha_2\colon A\to C$ in $\mathcal{K}_\fin$ there is $D\in \mathcal{K}_\fin$ together with $\beta_1\colon B\to D$ and $\beta_2\colon C\to D$ such that $\beta_1\circ\alpha_1=\beta_2\circ \alpha_2$.
    In that case, we say that $A$ is an {\it amalgamation base} in $\mathcal{K}_\fin$.
    \item[(SAP)] The {\it strong amalgamation property} is a strengthening of the AP, where we additionally have $\beta_1(B)\cap \beta_2(C)=\beta_1(\alpha_1(A)) (=\beta_2(\alpha_2(A))) $.  In that case, we say that $A$ is a {\it strong amalgamation base}. 
    \item[(CAP)] The {\it cofinal amalgamation property}:  any $A_0\in \mathcal{K}_\fin$ embeds into $A\in \mathcal{K}_\fin$ that is an amalgamation base in $\mathcal{K}_\fin$.
    \item[(WAP)] The \emph{weak amalgamation property}:
    for every $A_0 \in \K_\fin$ there is an embedding $i\maps A_0 \to A$ in $\K_\fin$ that is an \emph{amalgamable map}, i.e. for any $\alpha_1\colon A \to B$ and $\alpha_2\colon A\to C$ in $\mathcal{K}_\fin$ there is $D\in \mathcal{K}_\fin$ together with $\beta_1\colon B\to D$ and $\beta_2\colon C\to D$ such that $\beta_1\circ \alpha_1\circ i=\beta_2\circ \alpha_2\circ i$ holds.
\end{itemize}

We say $\K_\fin \neq \emptyset$ is a \emph{Fraïssé class} if it satisfies JEP and AP and has countably many isomorphism types.
Note that we do not require $\K_\fin$ to be hereditary.

\medskip

\begin{df} \label{def:Fraisse}
We say that a structure $U \in \K$ 
\begin{itemize}
    \item is \emph{universal} for $\K_\fin$  if for every $A \in \K_\fin$ there is an embedding $f\maps A \to U$;
    \item is \emph{homogeneous} for $\K_\fin$ if for every $A, B \in \K_\fin$ such that $A, B \leq U$ and every isomorphism $f\maps A \to B$ there is $h \in \Aut(U)$ extending $f$, or equivalently, for every $A \in \K_\fin$ and all embeddings $f, g\maps A \to U$ there is $h \in \Aut(U)$ such that $h \cmp g  = f$;
    \item has the \emph{extension property} for $\K_\fin$ if for every $A, B \in \K_\fin$ and embeddings $f\maps A \to U$ and $g\maps A \to B$ there is an embedding $h\maps B \to U$ such that $h \cmp g = f$.
\end{itemize}
Similarly we say that a sequence $(A_n, f_n)$ in $\K_\fin$ 
\begin{itemize}
    \item is \emph{cofinal} if for every $A \in \K_\fin$ there is an embedding $f\maps A \to A_n$ for some $n$,
    \item has the \emph{extension property} if for every $A, B \in \K_\fin$ and embeddings $f\maps A \to A_m$ and $g\maps A \to B$ there is an embedding $h\maps B \to A_n$ for some $n \geq m$ such that $h \cmp g = f^n_m \cmp f$,
    \item is \emph{absorbing} if for every $B \in \K_\fin$ and an embedding $g\maps A_m \to B$ there is an embedding $h\maps B \to A_n$ for some $n \geq m$ such that $h \cmp g = f^n_m$, i.e. it has the extension property for $f = \id{A_m}$,
    \item is \emph{Fraïssé} if it has the extension property (which under AP is equivalent to just being absorbing) and is cofinal.
\end{itemize}
\end{df}

The following two theorems are essentially well known.
For the classical setting of hereditary classes of one-sorted structures see Hodges~\cite[Theorems~7.1.2 and 7.1.7]{Hodges}.
For Theorem \ref{thm:Fclass}, a construction explicitly using the notion of a Fraïssé sequence can be found in \cite[Corollary~3.8]{Kub40} – we apply it to the special case of a category of finite structures, so the properties of being essentially countable, having a countable dominating subcategory, and having a cofinal subcategory are all equivalent to having countably many isomorphism types.
For a version of the theorem that states both implications in even more general setting see \cite[Theorem~4.46]{BartosKubis}.

\begin{tw} \label{thm:Fclass}
    $\K_\fin$ has a Fraïssé sequence if and only if it is a Fraïssé class.
\end{tw}

For the following theorem we sketch a proof.
A full proof in a more general setting can be found in \cite[Theorem~4.15]{BartosKubis}.
    
\begin{tw} \label{thm:Flim}
    Let $\K$ be a $\sigma$-complete category of structures, and let $U \in \sig\K_\fin$.
    The following conditions are equivalent.
    \begin{enumerate}
        \item $U$ is universal and homogeneous for $\K_\fin$.
        \item $U$ is universal and has the extension property for $\K_\fin$.
        \item $U$ is the colimit of a Fraïssé sequence in $\K_\fin$.
    \end{enumerate}
    Moreover, such object $U$ is unique, universal for $\sig\K_\fin$, and every sequence in $\K_\fin$ with colimit $U$ is Fraïssé.
\end{tw}
\begin{proof}[Proof sketch]
    The fact that $\K$ is $\sigma$-complete and that $\K_\fin$ consists of literally finite structures implies that there is an essentially one-to-one correspondence between sequences in $\K_\fin$ and structures in $\sig\K_\fin$ – including morphisms between them.
    The key property used is that every embedding from a $\K_\fin$-object to $U$ factorizes through the sequence – while this factorization property is trivial in the classical case, it is an important assumption in abstract category-theoretic setup.
    (For details see the notion of free completion \cite[Remarks~4.56 and 4.57]{BartosKubis}.)
    
    It follows that a sequence $(U_n, f_n)$ in $\K_\fin$ has the extension property if and only if its colimit $(U, f^\infty_n)$ in $\sig\K_\fin$ has the extension property for $\K_\fin$ (and similarly for universality).
    It also follows that two sequences having the same colimit are back-and-forth equivalent, and so if one sequence is Fraïssé, the other one is as well.
    
    It is easy to see that a universal homogeneous  structure for $\K_\fin$ has the extension property.
    Homogeneity of the colimit of a Fraïssé sequence as well as uniquess of a Fraïssé sequence (and its colimit) up to isomorphism follow from a back-and-forth argument.
\end{proof}

The unique object $U$ from Theorem~\ref{thm:Flim} is called the \emph{Fraïssé limit} of $\K_\fin$ (in $\K$), and by Theorem~\ref{thm:Fclass} it exists if and only if $\K_\fin$ is a Fraïssé class.

\medskip

\subsection{Transfer principles}

We formulate two propositions that will allow us to easily transfer  amalgamation principles and preserve the Fraïssé limit by suitable functors between categories.
This captures the situation that generalized reducts (images under forgetful functors) of a homogeneous structure are homogeneous in their respective categories.

First we need to introduce the relevant properties of functors.
We say that a functor $F\colon \C\to\D$  is
\begin{itemize}
    \item  {\it wide} / {\it surjective on objects} if for every $\D$-object $X$ there is a $\C$-object $X'$ such that $F(X')=X$;

  \item  {\it essentially wide} / {\it essentially surjective on objects} if for every $\D$-object $X$ there is a $\C$-object $X'$ with $F(X') $ isomorphic to $ X$;
		
	\item {\it cofinal} if for every $\D$-object $X$ there is a $\C$-object $X'$ and a $\D$-map $f\maps X \to F(X')$.
\end{itemize}

\begin{uwgi}
We have the following implications for a functor $F$:

\smallskip
\centerline{wide $\implies$ essentially wide $\implies$ cofinal.}
\end{uwgi}

\noindent
We say that a functor $F\colon \C\to\D$  is
\begin{itemize}
\item {\it star-surjective} if for every $\C$ object $X'$ and every $\D$-map $f\colon F(X')\to Y$ there is a $\C$-object $Y'$ and a $\C$-map $f'\colon X'\to Y'$ with $F(f')=f$;

	\item {\it essentially star-surjective} if for every $\C$-object $X'$ and every $\D$-map $f\maps F(X') \to Y$ there is a $\C$-object $Y'$ and a $\C$-map $f'\maps X' \to Y'$ and an $\D$-isomorphism $g\maps Y \to F(Y')$ with $F(f') = g \circ f$;

   \item     {\it absorbing} if for every $\C$-object $X'$ and every $\D$-map $f\maps F(X') \to Y$ there is a $\C$-object $Y'$, a $\C$-map $f'\maps X' \to Y'$, and an $\D$-map $g\maps Y \to F(Y')$ with $F(f') = g \cmp f$;  
        
		\item {\it weakly absorbing} if for every $\C$-object $X'$ there is a $\D$-map $e\maps F(X') \to X$ such that for every $\D$-map $f\maps X \to Y$ there is a $\C$-object $Y'$, a $\C$-map $f'\maps X' \to Y'$, and an $\D$-map $g\maps Y \to F(Y')$ with $F(f') = g \circ f \circ e$ (wlog the map $e$ may be of the form $F(e')$ for some $\C$-map $e'\maps X' \to X''$).

\end{itemize}

\begin{uwgi}
We have the following implications for a functor $F$:

\smallskip
\centerline{star-surjective $\Rightarrow$ essentially star-surjective $\Rightarrow$ absorbing $\Rightarrow$ weakly absorbing.}
\end{uwgi}

\begin{uwgi}
    For every subcategory $\C \subs \D$ we have the corresponding inclusion functor, and the properties listed in this section simplify, e.g. a subcategory $\C \subs \D$ is absorbing if for every $\C$-object $X$, $\D$-object $Y$ and a $\D$-map $f\maps X \to Y$ there is a $\C$-object $Z$ and a $\D$-map $g\maps Y \to Z$ such that $g \cmp f$ is a $\C$-map.

    Similarly, when a sequence in $\D$ is viewed as a functor $\N \to \D$, the cofinality and absorption reduces to Definition~\ref{def:Fraisse}.
\end{uwgi}

Note that definitions of AP, CAP, and WAP make sense for arbitrary categories, and we use them in the following general proposition.

\begin{prop} \label{thm:fin_transfer}
    Let $\C$ and $\D$ be categories and $F\maps \C \to \D$ a functor.
    \begin{enumerate}
        \item If $\C$ has AP and $F$ is essentially wide and absorbing, then $\D$ has AP.
        \item If $\C$ has CAP and $F$ is cofinal and absorbing, then $\D$ has CAP.
        \item If $\C$ has WAP and $F$ is cofinal and weakly absorbing, then $\D$ has WAP.
    \end{enumerate}
    In particular, if $F$ is wide and star-surjective, all the amalgamation properties are transferred forward.
\end{prop}
\begin{proof}
    The proof is a straightforward diagram chasing.
    If $A'$ is an amalgamation base in $\C$ and $F$ is absorbing, then $A = F(A')$ is an amalgamation base in $\D$.
    Similarly, if $e'\maps A'_0 \to A'_1$ is an amalgamable map, then $e = F(e')$ is an amalgamable map.
\end{proof}

A functor $F\maps \C \to \D$ is {\it $\sigma$-continuous} if  for every sequence $(X_n, f_n)$ in $\C$ and its colimit $(X_\infty, f^\infty_n)$ we have that $(F(X_\infty), F(f^\infty_n))$ is a colimit of $(F(X_n), F(f_n))$ in $\D$.
For categories of structures, where objects are (many-sorted) structures and morphisms are all embeddings, the colimit of a sequence is essentially the union of the structures, and so to check $\sig$-continuity, it is enough to show that for every point $x \in F(X_\infty)$ there is $n \in \omega$ and $x' \in F(X_n)$ such that $f_n^\infty(x') = x$.

\begin{prop} \label{thm:Flim_transfer}
    Let $\C$ and $\D$ be $\sigma$-complete categories of structures, and let $F\maps \C \to \D$ be a $\sigma$-continuous functor that restricts to $F_\fin\maps \C_\fin \to \D_\fin$.
    If $\C_\fin$ is Fraïssé with limit $U$ and
    \begin{enumerate}
        \item $F_\fin$ is essentially wide and absorbing, or
        \item $F_\fin$ is cofinal and absorbing, and $\D_\fin$ has AP,
    \end{enumerate}
    then $\D_\fin$ is Fraïssé with limit $F(U)$.
\end{prop}

\begin{proof}
Since $U \in \sig\C_\fin$, it is a $\C$-colimit of a $\C_\fin$-sequence $(U_n, f_n)$, which is necessarily Fraïssé by Theorem~\ref{thm:Flim}.
By Proposition~\ref{thm:fin_transfer}(1) applied to $F_\fin$, the case (1) reduces to the case (2), which we now show.

By straightforward diagram chasing, $(F(U_n), F(f_n))$ is cofinal in $\D_\fin$ as $(U_n, f_n)$ is cofinal in $\C_\fin$ and $F_\fin$ is cofinal.
Similarly, $(F(U_n), F(f_n))$ is absorbing in $\D_\fin$ as $(U_n, f_n)$ is absorbing in $\C_\fin$ and $F_\fin$ is absorbing.
Altogether, $(F(U_n), F(f_n))$ is a Fraïssé sequence in $\D_\fin$.

Since $F$ is $\sig$-continuous, $F(U)$ is the $\D$-colimit of $(F(U_n), F(f_n))$, and so it is the Fraïssé limit of $\D_\fin$.
\end{proof}

\subsection{The class of all finite ultrametric spaces}

We will show that $\Ult_\fin$ is a Fraïssé class and describe its Fraïssé limit $\U$.
To show the amalgamation property we will first assume that we work with a fixed ambient distance set $E \in \Lin^0$ and with isometric embeddings, i.e. in $\UltIso^E_\fin$.
Then we will use a transfer principle to generalize to dc-embeddings.

The following are straightforward generalizations of amalgamation of one-point extensions of classical ultrametric spaces.

\begin{lm} \label{thm:one_point_amalgamation}
    Let $E \in \Lin^0$ and let $X \in \UltIso^E_\fin$.
   Any two one-point extensions of $X$ are strongly amalgamable, i.e. for all one-point extensions $X \cup \set{a}, X \cup \set{b} \in \UltIso^E_\fin$ with $a \neq b$ there is $d(a, b) \in E \setminus \set{0}$ such that $X \cup \set{a, b} \in \UltIso^E_\fin$.
    
    Moreover, if the extensions are non-isomorphic, i.e. if $d(a, x_0) \neq d(b, x_0)$ for some $x_0 \in X$, then such $d(a, b)$ is unique.
\end{lm}
\begin{proof}
    We put $Y = X \cup \set{a, b}$ and $D_Y = D_{X \cup \set{a}} \cup D_{X \cup \set{b}}$, which is finite.
    Since $d$ is already specified on $X \cup \set{a}$ and $X \cup \set{b}$, we just need to specify a suitable $d(a, b)$.

    Suppose without loss of generality $s=d(a,x_0) < d(b,x_0)=r$ for some $x_0 \in X$. We are forced to define $d(a,b) = r$ in order to $a, b, x_0$ become an ultrametric triangle.
    It still needs to be checked that all triangles $a, b, x$ with $x \in X$ are ultrametric.
    This involves a case analysis in the tetrahedron $a, b, x_0, x$ depending on how $d(b, x)$ compares to $r$, which is left to the reader.

    Now suppose that $d(a, x) = d(b, x)$ for every $x \in X$.
    We put $d(a, b) = \min(D_Y \setminus \set{0})$.
    Then for every $x \in X$ we have $d(a, b) \leq d(a, x) = d(b, x)$, so $a, b, x$ is an ultrametric triangle.
\end{proof}

Note the category of distance sets $\Lin^0$ is equivalent to the category of linear orders $\Lin$, and so $\Lin^0_\fin$ is Fraïssé and its limit is $\Qpos$.
The functor $D_\fin\maps \Ult_\fin \to \Lin^0_\fin$ is clearly wide and absorbing (since distances do not have to be attained), and $D\maps \Ult \to \Lin^0$ is $\sig$-continuous.
Hence, by the transfer principle Proposition~\ref{thm:Flim_transfer}(1), if $\Ult_\fin$ is Fraïssé with limit $\U$, then $D_\U$ is isomorphic to $\Qpos$.
Therefore we focus on $\UltIso^E$ with $E = \Qpos$.

In order to use the transfer principles to extend amalgamation to dc-embeddings, we need to prove the following.
\begin{lm}\label{lm:inclusion}
    The inclusion functor $\UltIso^{\Qpos}_\fin \subs \Ult_\fin$ is essentially wide and essentially star-surjective.
\end{lm}
\begin{proof}
    Essentially we just observe that abstract distances can be identified with rational ones.
    Specifically, let $X \in \UltIso^{\Qpos}_\fin$, $Y \in \Ult_\fin$ and let $f\maps X \to Y$ be a dc-embedding.
    By the extension property of $\Qpos$ over $\Lin^0_\fin$ there is an $\Lin^0$-isomorphism $\phi\maps D_Y \to E \subs \Qpos$ such that $\phi \cmp D_f$ is the inclusion $D_X \subs E$.
    We put $Y' = \seq{Y, d'_Y, E}$ where $d'_Y(y, y') = \phi(d_Y(y, y'))$.
    Then $Y' \in \UltIso^{\Qpos}_\fin$, 
    and $g = \seq{\id{Y}, \phi}\maps Y \to Y'$ is a dc-isomorphism such that $g \cmp f$ is an isometric embedding.
    This shows that the inclusion $\UltIso^{\Qpos}_\fin \subs \Ult_\fin$ is essentially star-surjective. 
    
    We obtain essential surjectivity on objects as well since when we start with $Y \in \Ult_\fin$ and $X = \emptyset$ with $D_X = \set{0}$, by the proof above, we obtain $Y' \in \UltIso^{\Qpos}_\fin$ with $Y' \cong Y$.
\end{proof}

Since every two-sorted ultrametric space $X$ can be viewed both as a member of $\Ult$ and $\UltIso^E$ (for any $E \in \Lin^0$ with $D_X \leq E$), we say that $X$ is \emph{dc-homogeneous} if it is homogeneous for $\Ult_\fin$, and that $X$ is \emph{iso-homogeneous} if it is homogeneous for $\UltIso^E_\fin$.

Recall that the class of classical finite rational ultrametric spaces with isometric embeddings, which in our setup is isomorphic to the category $(\UltIso^{\Qpos}_{\rm prec})_\fin$ (see Section \ref{prel:2sort}), is Fraïssé and that its Fraïssé limit is the \emph{countable rational Urysohn ultrametric space} $\U_\Q$.

\begin{tw} \label{thm:U_Fraisse}
    $\Ult_\fin$ is a Fraïssé class, and its Fraïssé limit $\U$ is dc-isomorphic to $\U_\Q$.
    In particular, $\U$ is both dc-homogeneous and iso-homogeneous.
\end{tw}
\begin{proof}
    The class $\UltIso^{\Qpos}_\fin$ has the amalgamation property by Lemma~\ref{thm:one_point_amalgamation} since it is enough to check it for one-point extensions.
    It follows that $\UltIso^{\Qpos}_\fin$ is a Fraïssé class.  Indeed, clearly $\UltIso^{\Qpos}_\fin$ has countably many isomorphism types.
    Moreover, the empty space $\seq{\emptyset, \set{0}}$ uniquely embeds into every ultrametric space, and so JEP follows from AP.

    We apply Proposition~\ref{thm:Flim_transfer}(2) to the inclusion functor of the full subcategory of precise spaces $(\UltIso^{\Qpos}_{\rm prec}) \subs \UltIso^{\Qpos}$ to show that the Fraïssé limit of $\UltIso^{\Qpos}_\fin$ is $\U_\Q$.
    By Lemma~\ref{precise}, the subcategory $\monster \subs \UltIso^{\Qpos}_\fin$ is cofinal, and since it is also full, it is absorbing.
    Clearly, the inclusion functor is $\sig$-continuous.
    
    Finally we apply Proposition~\ref{thm:Flim_transfer}(1) to the inclusion functor $\UltIso^{\Qpos} \subs \Ult$, which is also clearly $\sig$-continuous and whose restriction $\UltIso^{\Qpos}_\fin \subs \Ult_\fin$ is essentialy wide and absorbing by Lemma~\ref{lm:inclusion}.
    Hence, $\Ult_\fin$ is Fraïssé, and the Fraïssé limit is represented by $\U_\Q$.
\end{proof}

\begin{uwgi} \label{rmk:echeloned}
The alternative description of echeloned spaces~\cite[Section~6]{GPPPS_Echeloned} exactly corresponds to precise two-sorted ultrametric spaces with the same notion of embeddings, but with the crucial difference that no form of triangle inequality is assumed.
Consequentially, the Fraïssé limits are different: as we showed the generic ultrametric space $\U$ is represented by the countable rational Urysohn ultrametric space, while the generic echeloned space is represented by the countable $S$-valued Urysohn \emph{metric}
space for $S = \set{0} \cup ((1, 2) \cap \Q)$.
Note that this is a \emph{dull} metric space~\cite[Definition~1.9]{GPPPS_Echeloned} in the sense that the distance set itself assures that the triangle inequality is trivially true.
\end{uwgi}

\subsection{Convexly ordered ultrametric spaces}

A \emph{convex order} on an ultrametric space $X$ is a liner order $\preceq$ on $X$ such that all balls are convex, i.e. for every ball $B \subseteq X$ and $x, y \in B$ the interval $[x, y]_{\preceq}$ is contained in $B$.
It follows that for disjoint balls $B$ and $B'$ the truth of $x \prec x'$ does not depend on the choice of $x \in B$ and $x' \in B'$.
In particular, we obtain a linear order of every adjacency equivalence class.
On the other hand, specifying a linear order $\preceq$ of every adjacency equivalence class gives a convex liner order of $X$ by putting $x \prec x'$ if $B(x,x')\prec B(x',x)$ for $x \neq x'$, i.e. the order is decided by the order of the adjacent envelope.
This can be also observed from the tree representation (Section~\ref{sec:treerep}), i.e. by linearly ordering the set of all branching classes at every node.
Let $\UltConv$ denote the category of convexly ordered ultrametric spaces and dc-embeddings preserving the convex order.

We consider the functor $F^\prec\maps \UltConv \to \Ult$ forgetting the convex order, and its composition $D^\prec\maps \UltConv \to \Lin^0$ with the functor $D\maps \Ult\to\Lin^0$.
Let $\IsoConv = \UltIso^{\Qpos\mkern-4mu, \prec} \subs \Ult^{\prec}$ denote the subcategory of all convexly ordered ultrametric spaces~$X$ with $D_X \leq \Qpos$ and all isometric embeddings, preserving the convex order.
We summarize the situation in the following diagram.

\vspace{-2ex}
\begin{center}
\begin{tikzpicture}[
    x = {(6em, 0em)},
    y = {(0em, -4em)},
    label/.style = {
        edge label={#1},
        font=\footnotesize,
    },
]
    \node (I) at (0, 0) {$\UltIso^{\Qpos}$};
    \node (U) at (1, 0) {$\Ult$};
    \node (CI) at (0, 1) {$\IsoConv$};
    \node (CU) at (1, 1) {$\UltConv$};
    \node (L) at (2, 0) {$\Lin^0$};

    \graph{
        (CU) ->[label=$F^\prec$, swap] (U),
        (CI) ->[label=$F^\prec$, swap] (I),
        (I) ->[label=$\subs$] (U) ->[label=$D$] (L),
        (CI) ->[label=$\subs$] (CU) ->[label=$D^\prec$, swap, inner sep=1pt] (L),
    };
\end{tikzpicture}
\end{center}

The following results show that the properties of $\UltIso^{\Qpos} \subs \Ult$ are mirrored in $\IsoConv \subs \UltConv$.

\begin{lm} \label{thm:C_one_point_amalgamation}
    Let $X \in \IsoConv_\fin$.
    Every two one-point extensions of $X$ are strongly amalgamable, i.e. for all one-point extensions $X \cup \set{a}, X \cup \set{b} \in \IsoConv_\fin$ with $a \neq b$ there is $d(a, b) \in \Q^{>0}$ and a choice of $a \prec b$ or $a \succ b$ such that $X \cup \set{a, b} \in \IsoConv_\fin$.
\end{lm}
\begin{proof}
    From Lemma~\ref{thm:one_point_amalgamation} we have $X \cup \set{a, b} \in \UltIso^{\Qpos}_\fin$. The linear orders on adjacency equivalence classes of $X \cup \{a\}$ and $X \cup \{ b\}$ combine into  partial orders on adjacency equivalence classes of $X \cup \{a, b\}$, which we extend to  linear orders arbitrarily.
    These linear extensions correspond to a  convex order $\preceq$ extending $\preceq_{X \cup \set{a}} \cup \preceq_{X \cup \set{b}}$.
\end{proof}

\begin{lm}\label{lm:C_inclusion}
    The inclusion functor $\IsoConv_\fin \subs \Ult^\prec_\fin$ is essentially wide and absorbing.
\end{lm}
\begin{proof}
    Since the distance sets do not interact with the convex order, we can follow the proof of Lemma~\ref{lm:inclusion} line by line.
\end{proof}

\begin{lm}\label{thm:forget_convex}
    The forgetful functor $F^\prec_\fin\maps \Ult^\prec_\fin \to \Ult_\fin$ is wide and star-surjective.
\end{lm}
\begin{proof}
    $F^\prec_\fin$ is wide as any ultrametic space can be endowed with a convex order.
    $F^\prec_\fin$ is star-surjective as for any ultrametric spaces $X \leq Y$ any convex order on $X$ can be extended to a convex order on $Y$ – we just linearly extend the induced partial orders on adjacency equivalence classes of $Y$ as in the proof of Lemma~\ref{thm:C_one_point_amalgamation}.
\end{proof}

\begin{tw} \label{thm:CU_Fraisse}
    $\IsoConv_\fin$ and $\Ult^\prec_\fin$ are Fraïssé classes whose Fraïssé limits can be identified. Denoting this  Fraïssé limit by $\U^\prec$, we obtain that $F^\prec(\U^\prec)$ is dc-isomorphic to $\U$.
    In other words, $\U$ can be endowed with a generic convex order.
\end{tw}
\begin{proof}
    The first part is proved analogously to the proof of Theorem~\ref{thm:U_Fraisse} with the use of Lemmas~\ref{thm:C_one_point_amalgamation} and \ref{lm:C_inclusion}.
    For the second part we again use the transfer principle (Proposition~\ref{thm:Flim_transfer}(1)) applied to $F^\prec$ with a help of Lemma~\ref{thm:forget_convex}.
\end{proof}

\begin{uwgi}
    The previous theorem again demonstrates a phenomenon when a single structure is the Fraïssé limit of multiple Fraïssé classes. Namely, $\U^\prec$ is simultaneously homogeneous for $\IsoConv_\fin$, $\UltIso^{\Qpos}_\fin$, $\Ult^\prec_\fin$, and $\Ult_\fin$.
\end{uwgi}

\subsection{Bounding the equilateral sets}

We will also consider ultrametric spaces with bounds for the cardinalities of equilateral sets.
A subset $S \subs X$ of an ultrametric space is \emph{$r$-equilateral} if $d(x, y) = r$ for every $x \neq y \in S$.
For every $x \in X$ and $r \in D_X \setminus \set{0}$, let $e_X(x, r) \in \N^{\geq 1} \cup \set{\infty}$ be the maximum cardinality of an $r$-equilateral subset of $X$ containing $x$ (where we do not distinguish between infinite cardinalities).
This is well-defined: $e_X(x, r)$ is either $1$ if the distance $r$ is not attained at $x$, or is equal to the number of adjacent balls of radius $r$ partitioning the precise ball $\bar{B}_r(x)$.

We say that an ultrametric space $X$ is \emph{$\delta$-bounded}, where $\delta$ is a map $E \setminus \set{0} \to \N^{\geq2} \cup \set{\infty}$ for some $D_X \leq E \in \Lin^0$, if $e_X(x, r) \leq \delta_X(r)$ for every $x \in X$ and $r \in D_X \setminus \set{0}$.

Now fix a set $M \subs \N^{\geq 2} \cup \set{\infty}$.
Let $\Ult^M$ denote the category whose objects are pairs $\seq{X, \delta_X}$ such that $X$ is a $\delta_X$-bounded ultrametric space for a map $\delta_X\maps D_X \setminus \set{0} \to M$, and whose morphisms are \emph{bounding-preserving} dc-embeddings $\seq{X, \delta_X} \to \seq{Y, \delta_Y}$, i.e. dc-embeddings $f\maps X \to Y$ such that $\delta_Y(D_f(r)) = \delta_X(r)$ for every $r \in D_X \setminus \set{0}$.

For a fixed distance set $E \in \Lin^0$ and bounding map $\delta\maps E \setminus \set{0} \to M$ we consider the full subcategory category $\UltIso^{E, \delta} \subs \UltIso^E$ of all $\delta$-bounded ultrametric spaces.
Clearly, every space $X \in \UltIso^{E, \delta}$ can be also viewed as an $\Ult^M$-object: with $\delta_X = \delta\restr{D_X \setminus \set{0}}$.
Hence, $\UltIso^{E, \delta} \subs \Ult^M$.

Let $\Lin^M$ denote the category of $M$-colored distance sets and color-preserving order-embeddings, i.e. an $\Lin^M$-object is a pair $\seq{E, \delta_E}$ with $E \in \Lin^0$ and $\delta_E\maps E \setminus \set{0} \to M$, and an $\Lin^M$-map $\seq{E, \delta_E} \to \seq{F, \delta_F}$ is an $\Lin^0$-map $f\maps E \to F$ such that $\delta_F(f(x)) = \delta_E(x)$ for every $x \in E$.
We have a functor $D^M\maps \Ult^M \to \Lin^M$ defined by $D^M(X) = \seq{D_X, \delta_X}$.
The categories and functors introduced are summarized in the following diagram.

\vspace{-2ex}
\begin{center}
\begin{tikzpicture}[
    x = {(6em, 0em)},
    y = {(0em, -4em)},
    label/.style = {
        edge label={#1},
        font=\footnotesize,
    },
]
    \node (I) at (0, 0) {$\UltIso^E$};
    \node (U) at (1, 0) {$\Ult$};
    \node (L) at (2, 0) {$\Lin^0$};
    \node (IM) at (0, 1) {$\UltIso^{E, \delta}$};
    \node (UM) at (1, 1) {$\Ult^M$};
    \node (LM) at (2, 1) {$\Lin^M$};
    
    \graph{
        (I) ->[label=$\subs$] (U) ->[label=$D$] (L),
        (IM) ->[label=$\subs$] (UM) ->[label=$D^M$] (LM),
        (IM) ->[label=$\subs$] (I),
        (UM) -> (U),
        (LM) -> (L),
    };
\end{tikzpicture}
\end{center}

It is easy to see that $\Lin^M_\fin$ is a Fraïssé class with Fraïssé limit $\Q^M = \seq{\Qpos, \delta_\gen}$ where $\delta_\gen\maps \Q^{>0} \to M$ is a \emph{generic coloring}, i.e. $\delta_\gen\fiber{m} \subs \Q^{>0}$ is order-dense for every $m \in M$.

\begin{lm} \label{thm:M_one_point_amalgamation}
    Let $X \in \UltIso^{E, \delta}_\fin$.
    Every two one-point extensions $X \cup \set{a}, X \cup \set{b} \in \UltIso^{E, \delta}_\fin$ can be amalgamated.
\end{lm}
\begin{proof}
    If the two extensions are isomorphic, i.e. $d(a, x) = d(b, x)$ for every $x \in X$, we just identify $a$ and $b$, so the amalgam is clearly $\delta$-bounded.
    Otherwise, without loss of generality we have $s = d(a, x_0) < d(b, x_0) = r$ for some $x_0 \in X$.

    By Lemma~\ref{thm:one_point_amalgamation}, putting $d(a, b) = r$ gives the unique amalgam $X \cup \set{a, b}$ in $\UltIso^E_\fin$.
    We show that $X \cup \set{a, b}$ is $\delta$-bounded.
    Any equilateral set that is not already contained in $X \cup \set{a}$ or $X \cup \set{b}$ is an $r$-equilateral set of the form $\set{a, b, y_1, \ldots, y_n}$.
    For every $i \leq n$ we have $d(a, x_0) = s < r = d(a, y_i)$, and so $d(x_0, y_i) = r$.
    Hence, $\set{x_0, b, y_1, \ldots, y_n} \subs X \cup \set{b}$ is an $r$-equilateral set of the same cardinality.
\end{proof}

\begin{lm} \label{lm:M_inclusion}
    The inclusion functor $\UltIso^{\Q^M}_\fin \subs \Ult^M_\fin$ is essentially wide and essentially star-surjective.
\end{lm}
\begin{proof}
    We proceed as in the proof of Lemma~\ref{lm:inclusion}.
    Given $X \in \UltIso^{\Q^M}_\fin$ and an $\Ult^M_\fin$-embedding $f\maps X \to Y$, we find an $\Lin^M$-isomorphism $\phi\maps \seq{D_Y, \delta_Y} \to \seq{D_{Y'}, \delta_{Y'}} \leq \seq{\Qpos, \delta_\gen}$ such that $\phi \cmp D_f$ is the inclusion $\seq{D_X, \delta_X} \leq \seq{D_{Y'}, \delta_{Y'}}$.
    This is possible since $\Q^M$ has the extension property for $\Lin^M_\fin$.
    By putting $Y' = \seq{Y, \phi \cmp d_Y, D_{Y'}, \delta_{Y'}}$ we obtain an $\Ult^M_\fin$-isomorphism $g = \seq{\id{Y}, \phi}\maps Y \to Y'$ such that $g \cmp f$ is an $\UltIso^{\Q^M}_\fin$-embedding.
    We prove essential wideness similarly.
\end{proof}

\begin{tw} \label{thm:MU_Fraisse}
    Let $M \subs \N^{\geq 2} \cup \set{\infty}$.    
    Then $\Ult^M_\fin$ is a Fraïssé class, its Fraïssé limit $\U^M$ is iso-homogeneous, and up to isomorphism we have $D^M(\U^M) = \seq{\Qpos,\delta_\gen}$.
    Moreover, $\U^M$ is precise and for every $x \in \U^M$ and $r \in \Q^{>0}$ we have $e_{\U^M}(x, r) = \delta_\gen(r)$.
\end{tw}

\begin{proof}
    As in the proof of Theorem~\ref{thm:U_Fraisse}, with the use of Lemma~\ref{thm:M_one_point_amalgamation}, Lemma~\ref{lm:M_inclusion}, and Proposition~\ref{thm:Flim_transfer}(1), we obtain that $\UltIso^{\Q^M}_\fin \subs \Ult^M_\fin$ are both Fraïssé classes with the same Fraïssé limit $\U^M$.
    Since $\U^M$ is Fraïssé limit of $\UltIso^{\Q^M}_\fin$ and since every isometry is bounds-preserving, we have that $\U^M$ is iso-homogeneous and $D^M(\U^M) = \seq{\Qpos, \delta_\gen}$.
    For every $x \in \U^M$, $r \in \Q^{>0}$ and finite $n \leq e_{\U^M}(x, r)$ we consider an $r$-equilateral space $Y = \set{x = y_1, y_2, \ldots, y_n}$.
    By the isometric extension property, $Y$ has an isomeric copy over $\set{x}$ in $\U^M$.
\end{proof}

Note that if $M = \set{\infty}$, there are no restrictions on equilateral sets and hence $\U^{\set{\infty}} = \U$.
The other extreme is $\U^{\set{2}}$, where no equilateral triangles are allowed.

From now on we just work with $\U$ rather than general $\U^M$.
We expect that most of results we prove for $\U$ generalize for spaces $\U^M$ (however, see Malicki~\cite{Mal}, where the cases $M = \set{\infty}$ and $M = \set{2}$ are special in the context of the structure of conjugacy classes in isometry groups).

Note that we cannot combine adding convex order and bounds on equilateral sets as this does not lead to an amalgamation class: for one-point extensions $\set{a, x} \sups \set{x} \subs \set{x, b}$ with $d(a, x) = d(x, b)$, the convex order $a \prec x \prec b$ forces us to create an equilateral triangle in every amalgamation.

\subsection{Properties of \texorpdfstring{$\clo{\U}$}{cl(U)} and its ball structure}\label{sec:cloU}

In Section \ref{sec:topandCau} we introduced Cauchy completions. In that notation $\clo{\U}$ is the Cauchy completion of $\U$; we can identify it with the Cauchy  completion of the rational Urysohn ultrametric space. In this section, we prove  universality of $\clo{\U}$ and we identify $\AutM(\clo{\U})$ as the automorphism group of the structure of balls of $\U$.

It follows directly from the fact that $\U$ is the Fraïssé limit of $\Ult_\fin$ (Theorem~\ref{thm:U_Fraisse}) and from general theory (Theorem~\ref{thm:Flim}) that every countable ultrametric space can be dc-embedded into $\U$.
This can be improved.
\begin{prop} \label{thm:universal_separable}
    For every countable ultrametric space $X$ there is a uniformly continuous dc-embedding $e\maps X \to \U$.
\end{prop}
\begin{proof}
    If $X$ is uniformly discrete, then every dc-embedding $X \to \U$ is uniformly continuous.
    Otherwise, $0$ is not isolated in $D_X$, and there is an $\Lin^0$-embedding $\phi\maps D_X \to \Qpos$ such that $\phi[D_X \setminus \set{0}] \subseteq \Q^{>0}$ is co-initial.
    Hence, by Lemma~\ref{thm:dc-continuity}, $f = \seq{\id{X}, \phi}\maps X \to X' = \seq{X, \phi \cmp d_X, \Qpos}$ is a uniformly continuous dc-embedding.
    We obtain $e$ as $g \cmp f$ where $g$ is any isometric embedding $X' \to \U$.
\end{proof}

\begin{wn}\label{wn:universal u bar}
    For every precise separable ultrametric space $X$ there is a uniformly continuous dc-embedding $e\maps X \to \clo{\U}$.
\end{wn}
\begin{proof}
    Let $X_0 \subseteq X$ be a countable dense subspace.
    By Proposition~\ref{thm:universal_separable}, there is a uniformly continuous dc-embedding $e_0\maps X_0 \to \U$.
    By Observation~\ref{exttoCauchy}, it extends to the uniformly continuous dc-embedding $\clo{e_0}\maps \clo{X_0} \to \clo{\U}$.
    Since $X$ is precise, by Lemma~\ref{thm:dense_embedding} we have $D_{X_0} = D_X$, and we can identify $\clo{X}$ with $\clo{X_0}$, and so there is a canonical isometric embedding $i\maps X \to \clo{X_0}$.
    Finally we put $e = \clo{e_0} \cmp i$.
\end{proof}

It turns out that the dc-automorphism group of $\clo{\U}$ with the metric pointwise convergence topology is Polish and can be identified with the automorphism group of the countable structure of balls $\Bee_\U$.
Moreover, $\Bee_\U$ is the Fraïssé limit of $\Ball_\fin$, i.e. it is a generic countable leveled adjacency tree.
We refer to Section~\ref{sec:adjballs} for the definition of  $\Bee_X$, the family of all adjacent balls of  an ultrametric space $X\in\Ult$, as well as for the definition of $\Ball$, the category of leveled adjacency trees, viewed as ball structures of ultrametric spaces.

\begin{prop} \label{thm:ball_automorphism_group}
    The map $\bar{\beta}\maps \AutM(\clo{\U}) \to \Aut(\Bee_{\U})$ is an isomorphism of Polish groups.
\end{prop}
\begin{proof}
    Since $\clo{\U}$ is precise and has no isolated points, the claim follows directly from Proposition~\ref{prop: beta}~(4).
\end{proof}

\begin{prop} \label{thm:ball_Fraisse_limit}
    $\Ball_\fin$ is Fraïssé with limit $\mathcal{B}_\U$.
\end{prop}
\begin{proof}
    We apply the transfer principle Proposition~\ref{thm:Flim_transfer}(1) to the ball functor $\Bee\maps \Ult \to \Ball$.
    The functor $\Bee$ is $\sig$-continuous as for every increasing chain $(X_n, f_n)$ with the union $X_\infty$ and an adjacent ball $B_{X_\infty}(x, y) \in \Bee_{X_\infty}$, there is $n \in \N$ such that $x, y \in X_n$, and so $B_{X_\infty}(x, y) = \Bee_{f^\infty_n}(B_{X_n}(x, y))$.
    By Lemma~\ref{thm:B_essentially_wide}, $\Bee_\fin$ is essentially wide.
    To show that $\Bee_\fin$ is absorbing let $X \in \Ult_\fin$ and $g\maps \Bee_X \to T$ be an $\Ball$-embedding.
    We take any $\Ball$-embedding $g'\maps T \to T'$ such that $T'$ has at least one node.
    By Lemma~\ref{thm:B_essentially_wide} there is $Y \in \Ult_\fin$ and a $\Ball$-isomorphism $i\maps T' \to \Bee_Y$.
    Finally, by Lemma~\ref{thm:B_star_surjective} there is a dc-embedding $f\maps X \to Y$ with $\Bee_f = i \cmp g' \cmp g$.
\end{proof}

\subsection{Concrete representations and completions}\label{section:concrete}

In this section we revisit a concrete representation of the rational Urysohn ultrametric space $\U$ and of its Cauchy completion $\clo{\U}$, and we describe a homogeneous spherical completion of $\U$ and its connection to valued fields.
Moreover, we use this concrete representation of $\U$ to obtain a continuous section to the canonical projection $\pi\maps \Aut(\U) \to \Aut(\Qpos)$, which we use in Section~\ref{sec:short_exact_sequence} to represent $\Aut(\U)$ as a semidirect product $\Iso(\U) \rtimes \Aut(\Qpos)$.

Let $S \ni 0$ be a set that will eventually represent splitting options in a tree.
In particular, it is natural to consider $S = \N$ or $S = \Q$.
By a support of a map $f\maps X \to S$ we mean the set $\supp(f) = \set{x \in X: f(x) \neq 0}$.

We consider the spaces
\begin{align*}
    V_S &= \set{f\maps \Q^{>0} \to S \mid \supp(f)\text{ is finite}}, \\
    \clo{V_S} &= \set{f\maps \Q^{>0} \to S \mid \supp(f)\text{ is finite or a decreasing sequence with limit }0}, \\
    \widehat{V_S} &= \set{f\maps \Q^{>0} \to S \mid \supp(f)\text{ is finite or a decreasing sequence}}, \\
    \clo{\clo{V_S}} &= \set{f\maps \Q^{>0} \to S \mid \supp(f)\text{ is a decreasing (countable) ordinal}}
\end{align*}
with the distance set $\Qpos$ and with the ultrametric
\[\textstyle
    d(f, g) = \max_{\Qpos}\set{r \in \Q^{>0}: f(r) \neq g(r)},
\]
i.e. as $r$ decreases, $f$ and $g$ agree up to $r = d(f, g)$, where the functions split.

If $S$ is linearly ordered, we define a convex order by:
\[
    f \prec g \iff f(r) < g(r) \quad\text{ for } r = d(f, g),
\]
and we denote the corresponding expansions by $V_S^\prec$, etc.

We also define the trees of codes of balls:
\begin{align*}
    T_S &= \set{t\maps [r_t, \infty)_{\Q} \to S \mid r_t \in \Q^{>0}}, \\
    \bar{T}_S &= \set{t\maps (r_t, \infty)_{\Q} \to S \mid r_t \in \Q^{>0}}.
\end{align*}
It is easy to see that for every space $X \in \set{V_S, \clo{V_S}, \widehat{V_S}, \clo{\clo{V_S}}}$, $f \in X$, and $r \in \Q^{>0}$ we have
\begin{align*}
    B_r(f) &= \set{g \in X: g \sups t} & &\text{for $t = f\restr{[r, \infty)} \in T_S$}, \\
    \bar{B}_r(f) &= \set{g \in X: g \sups t}  = \textstyle\bigcup_{s \in S} B_r(f \cup \set{r \mapsto s}) & &\text{for $t = f\restr{(r, \infty)} \in \bar{T}_S$},
\end{align*}
so elements of $T_S$ and $\bar{T}_S$ indeed act as codes of balls in all the spaces considered.

\begin{uwgi} \label{rmk:representations}
    Various versions of the above representation appear in literature, though mostly in isometric context.
    One can be more general and replace $\Qpos$ with an arbitrary distance set $E \in \Lin^0$ and obtain spaces $V_{E, S}$, $\clo{V_{E, S}}$, $\widehat{V_{E, S}}$, and $\clo{\clo{V_{E, S}}}$.
    The spaces $V_{E, \Q}$ and $\clo{V_{E, \Q}}$ for countable $0 \in E \subseteq \Rpos$ appear in Nguyen Van Thé~\cite[Sections 1.3.2 and 1.4.2]{Nmem}, along with more references.
    Spaces $V_{E,S}$ with an abstract set of distances were already studied from a model-theoretic point of view by Delon~\cite{Delon84}.
    See also Delhomme--Laflamme--Pouzet--Sauer~\cite{DLPS} for more on spaces $\clo{\clo{V_{E, S}}}$ for $E \subseteq \Rpos$ and their iso-homogeneity.

    The universal spaces from the introduction also fit this representation.
    The generalized Vestfrid's space $U_E$ of $E$-valued (non-strictly) decreasing sequences converging to zero, where $E \subseteq \Rpos$ and $d((x_n), (y_n)) = \max\set{x_k, y_k}$ for $k = \min\set{n \in \N: x_n \neq y_n}$, is isometric to $\clo{V_{E, \N}}$:
    Every such sequence $(x_n) \in E^\N$ corresponds to the map $f\maps E^{>0} \to \N$ defined by $f(r) = \card{\set{n \in \N: x_n = r}}$.
    Vice versa, for $f \in \clo{V_{E, \N}}$ we consider the sequence $(x_n)$ that is the decreasing enumeration of $\supp(f)$ where every value $r \in \supp(f)$ is repeated $f(r)$-many times (and where infinitely many zeros are added to the end if $\supp(f)$ is finite).
    It can be checked that the specified correspondence is indeed an isometry between $U_E$ and $\clo{V_{E, \N}}$.

    As for the Lemins' spaces, we have $\clo{\clo{V_{\Qpos, \kappa}}} \subseteq L_\kappa$.
    The subtlety is that the distance formula for $L_\kappa$ uses supremum instead of maximum: $d(f, g) = \sup_{\Rpos}\set{r \in \Q^{>0}: f(r) \neq g(r)}$.
    So instead of requiring the functions $f\maps \Q^{>0} \to \kappa$ to have co-well-ordered support, it is enough that they are eventually zero.
    On the other hand, this means that the distance set extends to $\Rpos$, and also that $\kappa$ does not any more correspond to the splittings in the tree representation.
    In fact, the splitting at $t\maps (r_t, \infty) \to \kappa$ corresponds to $\set{e\maps (0, r_t]_\Q \to \kappa}/{\sim}$, where $\sim$ denotes the equivalence by eventual agreement.
    In particular, $\clo{\clo{V_{\Qpos, \N}}} \subseteq L_\N \hookrightarrow \clo{\clo{V_{\Rpos, \mathfrak{c}}}}$.
\end{uwgi}

\begin{prop} \label{thm:section}
    For every space $X \in \set{V_S, \clo{V_S}, \widehat{V_S}, \clo{\clo{V_S}}}$ the canonical projection $\pi_X\maps \Aut(X) \to \Aut(\Qpos)$ admits a homomorphism $s_X\maps \Aut(\Qpos) \to \Aut(X)$ such that $\pi_X \cmp s_X = \id{\Aut(\Qpos)}$.
Moreover, $s_{X}\maps \Aut(\Qpos) \to \Aut(X)$ is continuous for $X=V_S$, and $s_{X}\maps \Aut(\Qpos) \to \AutM(X)$ is continuous for $X=\clo{V_S}$. 
\end{prop}
\begin{proof}
    We define $s_X(h)(f) = \seq{f \cmp h^{-1},h}$.
    Clearly, if $h\in \Aut(\Qpos)$ and $f\in X$, then $f \cmp h^{-1}\in X$, and $s_X(h)$ is a bijection.
    Moreover for every $f, g \in X$ we have,
    \begin{align*}
        d(s_X(h)(f), s_X(h)(g)) &= \max\set{x\colon f(h^{-1}(x))\neq g(h^{-1}(x))} \\ &= h(\max\set{x\colon f(x) \neq g(x)}) = h(d(f, g)),
    \end{align*}
    i.e. $s_X(h)$ is a dc-embedding with $D_{s_X(h)} = h$, which concludes that $s_X(h) \in \Aut(X)$.
    Furthermore, $s_X$ is a homomorphism, and we have $\pi_X \cmp s_X = \id{\Aut(\Qpos)}$.
     
    We now show that $s_X$, for $X=\clo{V_S}$, is continuous.
    Fix $F = \set{f_1,\ldots, f_n} \subs X$ and $q\in \Q^{>0}$, and consider the corresponding open neighbourhood of the identity 
     \[ U= N_{F, q}(\id{X}) = \{g\in \Aut(X)\colon d(g(f_i),f_i)< q,\ i=1,\ldots, n\}\]
    in $\Aut(X)$. Let 
    \[P=\{a\in \Q\colon a>q \text{ and for some } i,\ f_i(a)\neq 0\}\cup\{q\},
    \]
    and let $W = \Aut(\Qpos)_P$ be the pointwise stabilizer of $P$.
    Clearly $P$ is finite and hence $W$ is an open neighbourhood of the identity in $\Aut(\Qpos)$. Moreover, as $q\in P$, for any $h\in W$ we have $h([q,\infty))=[q,\infty)$. This implies that for any $i=1,\ldots, n$ and $h\in  W$, $(f_i\circ h^{-1})\restriction_{[q,\infty)}=f_i\restriction_{ [q,\infty)}$, and so
    $s_X[W]\subseteq U$, which finishes the proof of continuity of $s_X$.
    
    For $X=V_S$ the proof of continuity of $s_X$ is essentially the same with $ U=N_{F}(\id{X}) = \{g\in \Aut(X)\colon d(g(f_i),f_i)=0,\ i=1,\ldots, n\}$
    and $q=0$ in all other places. 
\end{proof}

\begin{lm} \label{thm:dc_from_iso}
    Let $X$ be an iso-homogeneous ultrametric space with $D_X = \Qpos$.
    If the canonical projection $\pi\maps \Aut(X) \to \Aut(\Qpos)$ is surjective, then $X$ is dc-homogeneous.
\end{lm}
\begin{proof}
    Let $p\colon A\to B$ be a finite partial dc-automorphism  of $X$. Take any $h\in \aut(\qplus)$ extending $D_p$ and let $f \in \Aut(X)$ be such that $D_f = h^{-1}$.
    Let $q\maps B \to f[B]$ be the restriction of $f$.
    Note that $q \cmp p$ is an isometry $A \to f[B]$, and so it can be extended to some $g\in\Iso(\U)$.
    Then $f^{-1} \cmp g$ is a dc-automorphism extending~$p$.
\end{proof}

Besides the Cauchy completion we shall also consider spherical completions.
Recall that an ultrametric space $X$ is said to be {\it spherically
complete}~\cite[Definition~5.1]{KirkShahzad} if every chain of balls in $X$ has nonempty intersection.
An ultrametric space $Y$ is said to be an {\it immediate extension}~\cite[page~34]{KirkShahzad} of an ultrametric space $X$ if $X\subseteq Y$ and if for each $y\in Y$ and $x\in X$ with $y\neq x$ there exists $x'\in X$ such that $d(x',y)< d(x,y)$.

It is known that an ultrametric space is spherically complete if and only if it has no proper immediate extension and that the relation of being an immediate extension is transitive (see \cite{Schorner} or \cite[7.3, 7.9]{PCR2}), hence a spherically complete immediate extension is the same thing as a maximal immediate extension and is called a \emph{spherical completion}.
Unlike the Cauchy completion, a spherical completion is not unique in general.

\medskip
We say that a ultrametric space $X$ is \emph{point iso-homogeneous} or \emph{point dc-homogeneous} if for every $x, y \in X$ there is $f \in \Iso(X)$ or $\Aut(X)$, respectively, such that $f(x) = y$.
The following is known for real-valued ultrametric spaces, but the proof does not use any properties of the linearly ordered distance set.
\begin{tw}[{\cite[Corollary 2]{DLPS}}] \label{thm:point_iso-homogeneous}
    An ultrametric space $X$ is iso-homogeneous if and only if it is point iso-homogeneous.
\end{tw}

We are ready to summarize properties of our concrete representations.
\begin{prop} \label{thm:concrere}
    Let $S \ni 0$ be a any countable infinite set. 
    \begin{enumerate}
        \item $V_S$ is iso-homogeneous and dc-homogeneous, and it is a concrete representation of~$\U$.
        Moreover, $V_\Q^\prec$ is a concrete representation of the convexly ordered rational Urysohn ultrametric space $\U^\prec$. 
        
        \item $\clo{V_S}$ is the Cauchy completion of $V_S$, and so it is a concrete representation of~$\clo{\U}$.
        Moreover, $\clo{V_S}$ is both iso-homogeneous and dc-homogeneous, but it is not spherically complete.
        
        \item $\widehat{V_S}$ is a spherical completion of $V_S$, but it is not even point dc-homogeneous.
        
        \item $\clo{\clo{V_S}}$ is a spherical completion of $V_S$ that is both iso-homogeneous and dc-homogeneous.
    \end{enumerate}
\end{prop}

\begin{proof}

{\bf Concrete representations} in (1) and (2). To show that the countable space $V_S$ is a concrete representation of $\U$ it suffices to check that $V_S$ is universal and homogeneous for $\UltIso^{\Qpos}_\fin$. Both those properties will follow from the extension property. If $A\subseteq V_S$ is finite and $i\colon A\to B=A\cup\{v\}$ is an isometric embedding, let $r=\min\{d(v,a)\colon a\in A\}$, and let $A_0=\{a\in A\colon d(a,v)=r\}$. It is not hard to see that there is $v'\in V_S$ such that $d(v',a)=r$ for all $a\in A_0$. Indeed, take $v'\in V_S$ such that $v'(t)=a(t)$ for any $t>r$ and $a\in A_0$, $v'(t)=0$ for $t<r$, and $v'(r)$ is distinct from all $a(r)$, $a\in A_0$ -- this is possible as $S$ is infinite.  By the ultrametric inequality, we also have $d(v',a)=d(v,a)$ for all $a\in A$.

For $V_{\Qpos}^\prec$ and convexly ordered spaces $A^\prec, B^\prec=A^\prec\cup \{v\}$, we additionally require for $v'(r)$ that $\set{a \in A_0: a(r) < v'(r)} = \set{a \in A_0: a \prec_B v}$. Here we are using the fact that set of values $S = \Qpos$ is a dense linear order.

The fact that $\clo{V_S}$ is a concrete representation of $\clo{\U}$ follows from $\clo{V_S}$ being the Cauchy completion of $V_S$, which is proved below.

\medskip

   {\bf Homogeneity} in (1), (2), (4):
    We endow $S$ with an abelian group operation $+$ such that $0$ becomes the neutral element.
    Then every $X \in \set{V_S, \clo{V_S}, \clo{\clo{V_S}}}$ becomes a subgroup of the product group $S^{\Qpos}$.
    Indeed, if $x,a\in X$, then $x+a\in X$, as any subset of the union of the supports of two functions from $X$ is still an allowed support for a function in $X$. Moreover, for a fixed $a\in X$, the map $x\mapsto x+a$ is an isometry of $X$.
    Therefore, as for any $x,y\in X$ and $a=y-x$, we have $x+a=y$, $X$ is point iso-homogeneous. By Theorem  \ref{thm:point_iso-homogeneous}, $X$ is in fact iso-homogeneous.
    By Proposition~\ref{thm:section}, the canonical projection $\pi_X$ is surjective, and so by Lemma~\ref{thm:dc_from_iso}, $X$ is dc-homogeneous.

\medskip

    {\bf Non-homogeneity} in (3):
    Every point in $\widehat{V_S} \setminus \clo{V_S}$ is isolated, while every point in $\clo{V_S}$ is not isolated.
    Since dc-automorphisms (unlike dc-embeddings) are continuous (Lemma~\ref{thm:dc-continuity}), this show that $\widehat{V_S}$ is not point dc-homogeneous.

    \medskip

{\bf Cauchy completeness} in (2).
    For  $\{B_{r_n}(f_n)\colon r_n\to 0\}$, a chain of open balls in $\clo{V_S}$, where $(r_n)$ is a strictly decreasing sequence, we have $f_{n+1}\restr{[r_n, \infty)}=f_{n}\restr{[r_n, \infty)}$. Therefore
    $\bigcap_n B_{r_n}(f_n) = \{f\}$, where $f\restr{[r_n, \infty)}=f_{n}\restr{[r_n, \infty)}$ for all $n$, and so $f\in\clo{V_S}$, as the support of each $f_{n}\restr{[r_n, \infty)}$ is finite. On the other hand, any $f\in\clo{V_S}$ can be realized as the unique element in the intersection of a maximal chain of open balls in $V_S$. Therefore $\clo{V_S}$ is the Cauchy completion of $V_S$.

    \medskip

 {\bf  Spherical completeness} in (4) and (3): We first show it for $\clo{\clo{V_S}}$. Every chain of open balls in $\clo{\clo{V_S}}$ contains a cofinal subchain $\{B_{r_n}(f_n)\colon r_n\to q\}$, for some non-negative real number $q$, a strictly decreasing sequence $(r_n)$, and $f_n\in \clo{\clo{V_S}}$.
    Then $f_{n+1}\restr{[r_n, \infty)}=f_{n}\restr{[r_n, \infty)}$ for every $n$.
    Take $f$ such that for every $n$, $f\restr{[r_n, \infty)}=f_{n}\restr{[r_n, \infty)}$, and $f(t)=0$ for $t\leq\inf\{r_n\colon n\in\N\}$. Then $f\in \bigcap_n B_{r_n}(f_n)$. Since the support of each $f_n$ is co-well-ordered, the same is true for $f$, and so $f\in\clo{\clo{V_S}}$. This shows the spherical completeness of $\clo{\clo{V_S}}$.
  
   Spherical completeness of $\widehat{V_S}$ is proved essentially as for $\clo{\clo{V_S}}$.
    We have that either $\supp(f\restr{[r_n, \infty)})$ is finite for every $n$, or $S = \supp(f\restr{[r_{n_0}, \infty)})$ is an infinite sequence for some $n_0$ and $\supp(f\restr{[r_n, \infty)}) = S$ for every $n \geq n_0$.
    In both cases, $f \in \widehat{V_S}$.

\medskip
{\bf Non-spherical completeness} in (2):
    If instead we take a chain of open balls in $\clo{V_S}$, $\{B_{r_n}(f_n)\colon r_n\to q\}$, for some $q>0$, a strictly decreasing sequence $(r_n)$, and $f_n\in V_S$ with $f_n(r_n)\neq 0$, then any $g\in \bigcap_n B_{r_n}(f_n)$ will be such that $g(r_n)\neq 0$ for all $n$. As $(r_n)$ does not converge to $0$, $g\notin \clo{V_S}$. This implies that $\clo{V_S}$ is not spherically complete.
\medskip

 {\bf Immediate extensions} in (3) and (4):
    $\clo{\clo{V_S}}$ and so $\widehat{V_S}$ is an immediate extension of $V_S$: For any $f \in V_S$ and $g \in \clo{\clo{V_S}} \setminus V_S$ let $r = d(f, g)$.
    Hence, $f\restr{(r, \infty)} = g\restr{(r, \infty)}$, but $f(r) \neq g(r)$.
    Let $f'(t) = f(t)$ for $t>r$, $f'(r) = g(r)$, and $f'(t) = 0$ for $t < r$.
    Then $f' \in V_S$ and $d(f', g) < r$.
\end{proof}

\subsubsection{Connection to valued fields}

There is a natural connection between $\U$ and valued fields, as can be demonstrated using the concrete representations above.
A discussion of ultrametric spaces in the context of valued fields and their spherical completions is present for example in \cite{PGS} or \cite{Delon84}.

Recall that a \emph{value group} is an abelian group $\seq{\Gamma, +, 0}$ with a linear order $\leq$ that is preserved by addition.
Recall that a \emph{valued field} is a field $K$ together with a value group $\Gamma$ and a \emph{valuation} $v\maps K \to \Gamma \cup \set{\infty}$ ($+$ and $\leq$ are naturally extend to $\infty$) satisfying:
\begin{enumerate}
    \item $v(a) = \infty$ if and only if $a = 0$,
    \item $v(a \cdot b) = v(a) + v(b)$,
    \item $v(a + b) \geq  \min\{v(a), v(b)\}$.
\end{enumerate}
It is possible to re-interpret the value group $\Gamma$ as a distance set $D_\Gamma \in \Lin^0$ by passing to the multiplicative notation and by reversing the order.
We let $D_\Gamma$ be the set of all formal elements $e^{-x}$ for $x \in \Gamma \cup \set{\infty}$ (with $e^{-\infty} = 0$ and $e^{-0} = 1$) endowed with linear order and multiplication defined by 
\[
    e^{-x} \leq e^{-y} \iff x \geq y \text{\qquad and\qquad} e^{-x} \cdot e^{-y} = e^{-(x + y)}.
\]
The corresponding re-interpretation of the valuation $v$ as a norm $\abs{a} = e^{-v(a)}$ gives
\begin{enumerate}[label={(\arabic{*}')}]
    \item $\abs{a} = 0$ if and only if $a = 0$,
    \item $\abs{a \cdot b} = \abs{a} \cdot \abs{b}$,
    \item $\abs{a + b} \leq \max\set{\abs{a}, \abs{b}}$.
\end{enumerate}
Then putting $d(a, b) = \abs{a - b}$ gives a $D_\Gamma$-valued translation invariant ultrametric.
The translation goes both ways, i.e. a translation invariant ultrametric induces a norm satisfying (1') and (3'), and if it satisfies also (2'), it corresponds to a valuation.

In our case the value group $\Gamma$ is $\seq{\Q, 0, +, \leq}$ and the distance set $D_\Gamma$ is $\set{e^{-q}: q \in \Q \cup \set{\infty}} \subs \R$, which is (as a linear order) isomorphic to $\Qpos$.

Let $k$ be a countable infinite field.
Recall that $k[[t^\Q]]$ is the Hahn field consisting of formal power series $a = \sum_{q \in \Q} a_q t^q$ in variable $t$ with coefficients $a_q \in k$ such that $\supp(a) = \set{q \in \Q: a_q \neq 0}$ is well-ordered.
The operations are extending addition and multiplication of polynomials:
\begin{align*}
    \Bigl(\sum_{q \in \Q} a_q t^q\Bigr) + \Bigl(\sum_{q \in \Q} b_q t^q\Bigr) &= \sum_{q \in \Q} (a_q + b_q) t^q, \\
    \Bigl(\sum_{q \in \Q} a_q t^q\Bigr) \cdot \Bigl(\sum_{q' \in \Q} b_{q'} t^{q'}\Bigr) &= \sum_{q, q' \in \Q} a_q b_{q'} t^{q + q'}.
\end{align*}
The multiplication is well-defined because of well-ordered supports.
Inverse elements for multiplication exist and $k[[t^\Q]]$ is indeed a field.
The valuation is defined by 
\[\textstyle
    v(a) = \min_{\Q \cup \set{\infty}}(\supp(a)).
\]

\begin{prop}
    The ultrametric structure of $k[[t^\Q]]$ can be identified with $\clo{\clo{V_k}}$, and therefore $\clo{\clo{V_k}}$ can be endowed with a structure of a valued field.
\end{prop}
\begin{proof}
    Using the multiplicative reformulation of $k[[t^\Q]]$, we have
    \[\textstyle
        \abs{a} = \max_{D_\Q} (\supp(a))
        \qquad\text{and so}\qquad
        d(a, b) = \max_{D_\Q}\set{e^{-q}: q \in \Q, a_q \neq b_q}.
    \]
    We observe that this ultrametric matches the ultrametric of $\clo{\clo{V_k}}$, because the elements of $k[[t^\Q]]$ are functions $\Q \to k$ with well-ordered support, which can be identified with functions $D_\Q \setminus \set{0} \to k$ with decreasing well-ordered support, and $D_\Q \setminus \set{0} \cong \Q^{>0}$.
    Altogether, $k[[t^\Q]]$ and $\clo{\clo{V_k}}$ have mathching ultrametric structures, and hence the structure of a valued field can be transferred from $k[[t^\Q]]$ to $\clo{\clo{V_k}}$.
\end{proof}

In this view we have the following:

\begin{wn} \label{thm:correspondeces} \hfill
    \begin{enumerate}
        \item $V_k$ corresponds to the valued ring $k[t^\Q] = \set{a \in k[[t^\Q]]: \supp(a) \text{ is finite}}$.
        \item $\clo{V_k}$ corresponds to $\clo{k[t^\Q]} := \set{a \in k[[t^\Q]]: \supp(a)$ is finite or a sequence converging to $\infty}$, which is a subfield, and so $\clo{V_k}$ becomes a valued subfield of~$\clo{\clo{V_k}}$.
        \item $\widehat{V_k}$ is spherically complete, but is not closed under addition and under multiplication.
        \item The field of fractions $k(t^\Q)$ of $k[t^\Q]$ corresponds to a countable space $W_k$ such that $V_k \subs W_k \subs \clo{V_k}$, and hence $W_k$ is isometric to $V_k$.
    \end{enumerate}
\end{wn}
\begin{proof}
    Claim~(1) is clear because we restrict functions on both sides to those with finite support.

    Claim~(2): Again the correspondence is clear.
    Also clearly $\clo{k[[t^\Q]]}$ is an additive subgroup since the addition is defined coordinatewise.
    It is closed under multiplication since 
    $\supp(a \cdot b) \subs \bigcup_{q \in \supp(b)} (\supp(a) + q)$, and so $\supp(a \cdot b)$ is finite or a sequence converging to $\infty$ if the same is true for $a$ and $b$.
    
    If $a \in \clo{k[[t^\Q]]}$ is such that $\min(\supp(a)) = 0$ and $a_0 = 1$, i.e. $a = 1 + a_1 t^{q_1} + a_2 t^{q_2} + \cdots$ with $q_1 < q_2 < \cdots$, then the recursive procedure for computing $a^{-1}$ in $k[[t^\Q]]$ shows that $\supp(a^{-1})$ consists of finite sums of elements from $\supp(a)$, and therefore is again finite or a sequence converging to $\infty$.
    Hence $a^{-1} \in \clo{k[t^\Q]}$.
    Altogether, $\clo{V_k} \cong \clo{k[t^\Q]}$ is a valued field.
    
    Claim~(3): The spherical completeness follows from Propositon~\ref{thm:concrere}.
    Now let $K \subs k[[t^\Q]]$ be the subspace corresponding to $\widehat{V_k}$, i.e. it consists of functions whose support is finite or an increasing sequence.
    Clearly if for example $a = \sum_{n \in \N} t^{1 - 1/n}$, then $t \cdot a, (1 + t) \in K$, but $a + t \cdot a = (1 + t) \cdot a \notin K$ as its support is of type $\omega + \omega$.

    Claim~(4): We have $V_k \subs W_k \subs \clo{V_k}$ corresponding $k[t^\Q] \subs k(t^\Q) \subs \clo{k[t^\Q]}$ since $\clo{k[t^\Q]}$ is a field by (2) and $k(t^\Q)$ is the smallest field extension of $k[t^\Q]$.
    We have that $W_k$ is isometric to $V_k$ by Lemma~\ref{thm:countable_below_Cauchy} since $V_k$ is a concrete representation of $\U$ and $\clo{V_k}$ is its Cauchy completion.
\end{proof}

\begin{lm} \label{thm:countable_below_Cauchy}
    Let $M$ be a countable ultrametric space such that $\U \subs M \subs \clo{\U}$, where $\clo{\U}$ is the Cauchy completion of $\U$.
    Then $M$ is isometric to $\U$.
\end{lm}

\begin{proof}
    Let $A \subs M$ be finite, and let $A' = A \cup \set{x}$ be a one-point extension (using rational distances).
    Let $r \in \Qpos$ be strictly smaller than every distance used in $A'$.
    For every $a \in A$ we find a point $c(a) \in \U$ with $d(a, c(a)) \leq r$.
    Then $c\maps A \to c[A]$ is an isometry. By the extension property of $\U$ applied to $c[A] \cup \set{x}$, there is $y \in \U$ such that $d_{\U}(c(a), y) = d_{A'}(a, x)$ for every $a \in A$.
    But we also have $d_M(a, y) = d_M(c(a), y)$ since those distances are larger than $r$.
    Hence, the countable space $M$ has the extension property for one point extensions, and so is homogeneous.
\end{proof}

\begin{wn}\label{wn:Urysohn reduct}
    The countable rational Urysohn ultrametric space $\U$ is isometric to the reduct of the valued field $k(t^\Q)$ for any countable infinite field $k$.
\end{wn}
\begin{proof}
    By Corollary~\ref{thm:correspondeces}~(4) and by Proposition~\ref{thm:concrere} we have $k(t^\Q) \cong W_k \cong V_k \cong \U$.
\end{proof}

\begin{uwgi}
    $k[[t^\Q]]$ is a spherical completion of the valued field $k(t^\Q)$, and if $k$ is of characteristic zero, by Kaplanski~\cite[Theorem~5]{Kap}, the spherical completion is unique.
    Hence, we can view the ultrametric space $\clo{\clo{V}}$ as \emph{the} spherical completion of $V$ (note that $V \cong \U \cong k(t^\Q)$ as ultrametric spaces).
\end{uwgi}

\section{Automorphism groups}\label{Sec:automorphism}

\subsection{\texorpdfstring{Comparing $\Aut(\U)$ with $\Iso(\U)$ and $\Aut(\Qpos)$}
{Comparing Aut(U) with Iso(U) and Aut(Qpos)}}
\label{sec:short_exact_sequence}

In this section we prove  that $\aut(\U)$ is an  extension of $\aut(\qplus)$ by $\Iso(\U)$, and similarly that  $\aut(\U^\prec)$ is an  extension of $\aut(\qplus)$ by $\Iso(\U^\prec)$. Then we conclude that $\aut(\U^\prec)$ is extremely amenable, $\aut(\U)$ is amenable, and we identify the universal minimal flow of $\aut(\U)$.

\begin{df}
Let $G_1, G_2, H$ be topological groups. Let $i\colon G_1\to H$ be a continuous embedding, $\pi\colon H\to G_2$  a continuous and open surjection, 
and assume that the image of $i$ is equal to the kernel of $\pi$.
In that case we say that $H$ is an {\it extension} of $G_2$ by $G_1$. This is the same as saying that
 \[ 1 \xrightarrow{} G_1 \xrightarrow{i} H \xrightarrow{\pi} G_2   \xrightarrow{} 1\]
 is a {\it short exact sequence} of topological groups.
The sequence \emph{splits} if there exists a continuous section $s\maps G_2 \to H$, i.e. a continuous homomorphism such that $\pi \cmp s = \id{G_2}$.
\end{df}

Recall that $\Iso(\U)$ denotes the isometry group, i.e.  $f\in \Iso(\U)$ iff $D_f$
is the identity.
Let $\pi\maps \Aut(\U) \to \Aut(\Qpos)$ be the canonical projection, and let $i$ denote the inclusion
 $\Iso(\U)\subseteq \aut(\U)$. 
Note that $\Iso(\U) \leq \aut(\U)$ is a closed normal subgroup as it is the kernel of $\pi$.
We analogously introduce $\bar{\pi}\maps \AutM(\clo{\U}) \to \Aut(\Qpos)$ and $\bar{i}\maps \IsoM(\clo{\U}) \to \AutM(\clo{\U})$.

For every countable structure $M$ and its finite partial automorphism $p$, we put
\[
    V_p = V^M_p=\set{g \in \Aut(M): g\text{ extends }p}.
\]
We have $V_p = N_A(f)$ for any $f$ extending $p$ and for $A = \dom(p)$, so this is just a notation for a basic open set.
Similarly for an ultrametric space $M$, its precise finite partial automorphism $p$, and $q \in \Q^{>0}$ we put
\[
    V_{p, q} = V^M_{p, q} = \set{g \in \AutM(M): d(g(x), p(x)) < q\text{ for every } x \in \dom(p)},
\]
which is again equal to $N_{A, q}(f)$ for any $f$ extending $p$ and for $A = \dom(p)$.
(Recall the definition of a precise partial-dc automorphism and precise space from Section~\ref{prel:2sort} – $\seq{A, p}$ is precise if every distance of $A$, $\dom(p)$, and $\rng(p)$ is attained in the respective space.)

\begin{prop} \label{basic_image}
    We have \begin{enumerate}
        \item $\pi[V_p] = V_{D_p}$ for every $p$ finite partial automorphism of $\U$,
        \item $\bar{\pi}[V_{p, q}] = V_{D_p}$ for every $p$ precise finite partial automorphism of $\clo{\U}$ and $q \in \Q^{>0}$ such that $q \leq (D_{\dom(p)} \cup D_{\rng(p)}) \setminus \set{0}$.   
    \end{enumerate}
    In particular, the maps $\pi\maps \Aut(\U) \to \Aut(\Qpos)$ and $\bar{\pi}\maps \AutM(\clo{\U}) \to \Aut(\Qpos)$ are open.
\end{prop}
\begin{proof}
    We first prove (1).
    Take any $h\in V_{D_p}$.
    Suppose $f_0 \supseteq p$ is a precise finite partial dc-automorphism of $\U$ (and note that such extension exists since $p$ can be extended to a total dc-automorphism of $\U$ and since $\U$ is a precise space).
    We first show that for every $x\in \U\setminus\dom(f_0)$, we can extend $f_0$ to a precise finite partial dc-automorphism $f_1$ such that $x\in \dom(f_1)$ and $D_{f_1}$ still agrees with $h$.
    Let $A=\dom(f_0)$ and $B=\rng(f_0)$. We will construct $y\in \U$ such that setting $x\mapsto y$ we get the required~$f_1$. 
    
    Let $B' = B \cup \set{y'}$ for a new point $y'$.
    We extend the ultrametric from $B$ to $B'$ by putting $d(y', b) = h(d(x, f_0^{-1}(b)))$ for every $b \in B$.
    Then $f_0 \cup \set{x \mapsto y'}$ becomes a dc-isomorphism $A \cup \set{x} \to B'$ (which also shows that $B'$ is an ultrametric space).
    Since $\U$ is the Fraïssé limit of $\UltIso^{\Qpos}$, we can apply the isometric extension property to  $B\leq \U$ and $B\leq B'$ to obtain a point $y \in \U$ such that $\id{B} \cup \set{y' \mapsto y}$ is an isometry.
    Hence, $f_1 = f_0 \cup \set{x \mapsto y}$ is the desired dc-isomorphism $A \cup \set{x} \to B \cup \set{y}$.
    
    We have just shown how to extend $f_0$ to $f_1$ by adding an $x$ to the domain. To get an $f\in V_p$ such that $\pi(f)=h$ we start with $f_0=p$ and then  do a back-and-forth construction extending the domain at odd steps and the range at even steps. The union of partial functions constructed at all stages is the required~$f$.

    Now we show (2).
    We can identify $\Aut(\U)$ with the (algebraic) subgroup of $\Aut(\clo{\U})$, see Observation \ref{exttoCauchy}, so that $\bar{\pi}$ extends $\pi$. We want to show that $\bar{\pi}[V_{p, q}] = V_{D_p}$.
    The inclusion $\bar{\pi}[V_{p, q}] \subs V_{D_p}$ follows from Observation~\ref{thm:distance_neighborhood} since we assume that $q \leq (D_{\dom(p)} \cup D_{\rng(p)}) \setminus \set{0}$.

    On the other hand, for every $t \in \dom(p) \cup \rng(p)$ we fix a point $t' \in B(t, q) \cap \U$, and we put $p'(x') = y'$ whenever $p(x) = y$.
    Then $p'$ is a partial automorphism of $\U$ with $D_{p'} = D_p$ and $V_{p', q} = V_{p, q}$.
    Therefore, $V^{\clo{\U}}_{p, q} \supseteq V^\U_{p'}$, and so $\bar{\pi}[V_{p, q}] \supseteq \pi[V_{p'}] = V_{D_{p'}} = V_{D_p}$.
\end{proof}

\begin{wn}\label{short}
The sequences 
\begin{gather*}
1 \xrightarrow{} \Iso(\U) \xrightarrow{i} \aut(\U) \xrightarrow{\pi}
 \aut(\qplus) \xrightarrow{} 1, \\
1 \xrightarrow{} \IsoM(\clo{\U}) \xrightarrow{\bar{i}} \AutM(\clo{\U}) \xrightarrow{\bar{\pi}}
 \AutM(\qplus) \xrightarrow{} 1
\end{gather*}
are short exact.
\end{wn}
\begin{proof}
    Clearly $i$ is a continuous embedding, $\pi$ is continuous by Observation~\ref{thm:canonical_projection_continuous}, and the image of $i$ is equal to the kernel of $\pi$.
    By Proposition~\ref{basic_image}, $\pi$ is open, and since $\pi[\Aut(\U)] = \pi[V^\U_\emptyset] = V^{\Qpos}_\emptyset = \Aut(\Qpos)$, we have that $\pi$ is also surjective.
    The proof for $\bar{i}$ and $\bar{\pi}$ is completely analogous.
\end{proof}

Using the concrete representation of $\U$ obtained in Section \ref{section:concrete}, we show that $\Aut(\U)$ is the topological semidirect product $\Iso(\U) \rtimes  \aut(\qplus)$, from which Corollary~\ref{short} follows as well.

Recall that given topological groups $N$, $H$ and a homomorphism $\alpha\maps H \to \Aut(N)$ such that the corresponding map $H \times N \to N$ is continuous, we can form the topological semidirect product $N \rtimes_\alpha H$.
Topologically, we take the product $N \times H$, but the group operation and inverse are defined by
\[
    \seq{n, h} \cdot \seq{n', h'} = \seq{n \alpha_h(n'), hh'} \qquad \seq{n, h}^{-1} = \seq{n \alpha_{h^{-1}}(n^{-1}), h^{-1}}.
\]
Then $N \rtimes_\alpha H$ is a topological group, the corresponding projection $N \rtimes_\alpha H \to H$ is an open continuous homomorphism, whose kernel can be identified with $N$. The section $s\colon H\to N \rtimes_\alpha H$, $s(h)=\seq{1,h}$, is a continuous homomorphism.

Conversely, suppose that $\pi\maps G \to H$ is a continuous homomorphism of topological groups admitting a continuous section.
Note that it follows that $\pi$ open surjective since it is a homomorphism and a quotient map.
Let $N \leq G$ denote the kernel of $\pi$.
We consider the continuous action $\alpha_h(n) = s(h)n s(h)^{-1} \in N$ for $h \in H, n \in N$, and the map $\phi\maps N \rtimes_\alpha H \to G$ defined by $\phi(\seq{n, h}) = n s(h)$.
One can check that $\phi$ is an isomorphism of topological groups.
In this case we say that $\pi$ induces an isomorphism $G \cong N \rtimes H$ of topological groups. For a discussion on topological semidirect products, see also  \cite[Section 7.2]{DPS}.

\begin{tw} \label{thm:semidirect} \hfill
    \begin{enumerate}
        \item The canonical projection $\pi\maps \Aut(\U) \to \Aut(\Qpos)$ induces an isomorphism $\Aut(\U) \cong \Iso(\U) \rtimes \Aut(\Qpos)$.
        \item The canoninal projection $\bar{\pi}\maps \AutM(\clo{\U}) \to \Aut(\Qpos)$ induces an isomorphism $\AutM(\clo{\U}) \cong \IsoM(\clo{\U}) \rtimes \Aut(\Qpos)$.
    \end{enumerate}
    In particular, the short exact sequences $1 \to \Iso(\U) \to \Aut(\U) \to \Aut(\Qpos) \to 1$ and $1 \to \IsoM(\clo{\U}) \to \AutM(\clo{\U}) \to \Aut(\Qpos) \to 1$ split.
\end{tw}
\begin{proof}
    By Proposition~\ref{thm:concrere}, $V_\Q$ is a concrete representation of $\U$, and $\clo{V_\Q}$ is a concerete representation of $\clo{\U}$.
    By Proposition~\ref{thm:section}, the canonical projections admit  continuous sections, so we can use the theory above to get the conclusion.
\end{proof}

We now proceed to the class  $\ultraconf$ of convexly  ordered ultrametric spaces, we denoted 
 its Fraïssé limit by 
$\U^\prec$. Let $\Iso(\U^\prec)$ be the group of isometries of
$\U^\prec$, let $\pi^\prec\colon \aut(\U^\prec)\to \aut(\qplus)$ be the canonical projection, and let $i^\prec \colon \Iso(\U^\prec)\to \aut(\U^\prec)$ be the embedding. Then  $i^\prec[\Iso(\U^\prec)]$ is a closed subgroup of $\aut(\U^\prec)$.  We can identify $\Iso(\U^\prec)$
with the isometry group of the classical convexly ordered ultrametric Urysohn space $\U_\Q^\prec$, that is, the isometry group of the Fraïssé limit of the class of finite  convexly ordered ultrametric spaces with rational distances.

\begin{tw}\label{shortprec}
\[
1 \xrightarrow{} \Iso(\U^\prec) \xrightarrow{i^\prec} \aut(\U^\prec) \xrightarrow{\pi^\prec}
 \aut(\qplus) \xrightarrow{} 1
\]
is a short exact sequence. 
Moreover, the canonical projection $\pi^\prec$ induces an isomorphism $\Aut(\U^\prec) \cong \Iso(\U^\prec) \rtimes \Aut(\Qpos)$
\end{tw}
\begin{proof}
We proceed as in the proofs of Theorems~\ref{short} and \ref{thm:semidirect}.
The first part follows from the fact that $\ultraconf$ has the strong amalgamation property for isometries (Lemma~\ref{thm:C_one_point_amalgamation}).

In the second part we use that by Proposition~\ref{thm:concrere}, $V_\Q^\prec$ is a concrete representation of $\U^\prec$, and that 
the section $s = s_{V_{\Q}}\maps \Aut(\Qpos) \to \Aut(V_{\Q})$ defined in Proposition~\ref{thm:section} by $s(h)(f) = \seq{f \cmp h^{-1}, h}$ preserves the convex order, i.e. $s\maps \Aut(\Qpos) \to \Aut(V_{\Q}^\prec)$.
Indeed, it simply follows from the definition of the convex order:
$
f \prec g \iff f(x) < g(x) \text{ where } x = d(f, g) = \max\set{q \in \Qpos: f(q) \neq g(q)}.
$
\end{proof}

Let $G$ be a Hausdorff topological group. 
We say that $G$ is {\it extremely amenable} if  every continuous action of $G$ on a compact Hausdorff space has a fixed point.
It is {\it amenable} if every continuous affine action of $G$ on a compact convex subspace of a locally convex topological vector space admits a fixed point.

A {\it minimal flow} of  $G$ on a compact Hausdorff space $X$ is a continuous action of $G$ on $X$ such that all orbits are dense. We say that it is metrizable if $X$ is metrizable.
A minimal flow $M(G)$ of $G$ is a {\it universal minimal flow} if for any minimal flow $Y$ of $G$ there is a $G$-invariant continuous surjection $\rho\colon M(G)\to Y$. It is well known that for any topological Hausdorff group there is a unique up to isomorphism universal minimal flow.

We now apply results of Jahel--Zucker. Theorem \ref{jahzuc}(1) is Theorem 1.1 in \cite{JZ} and parts (2) and (3) of Theorem \ref{jahzuc} are sketched on page 15 of  \cite{JZ}.
 For a  proof of \ref{jahzuc}(2), see also
 Pestov~\cite[Corollary 6.2.10]{Pe2}.

\begin{tw}[Jahel--Zucker]\label{jahzuc}
Let \[ 1 \xrightarrow{} G_1 \xrightarrow{i} H \xrightarrow{\pi} G_2   \xrightarrow{} 1\]
be a short exact sequence of Polish groups. 
\begin{enumerate}
\item If both $G_1, G_2$  have metrizable universal minimal flows, then $H$ has metrizable universal minimal flow.
\item If both $G_1, G_2$  are extremely amenable, then $H$ is extremely amenable.
\item If both $G_1, G_2$  are amenable, then $H$ is amenable.
\end{enumerate}
\end{tw}

\begin{tw}\label{examen} \hfill
\begin{enumerate}
\item The Polish group $\aut(\U)$ has metrizable universal minimal flow.
\item The Polish group $\aut(\U^\prec)$ is extremely amenable.
\item The Polish group $\aut(\U)$ is  amenable.
\end{enumerate}
\end{tw}

\begin{proof}
    Pestov \cite[Theorem 5.4]{Pe} proved that $\aut(\Q)$, which is isomorphic to $\aut(\qplus)$, is extremely amenable. In particular, it is amenable and it has a metrizable universal minimal flow. 
    On the other hand, Nguyen Van Thé \cite[Theorem 20]{N} proved that the group $\Iso(\U^\prec)$ is extremely amenable and that $\Iso(\U)$ has  metrizable universal minimal flow \cite[Theorem 21]{N}. Therefore Theorem \ref{short} implies that $\aut(\U)$ has metrizable universal minimal flow and Theorem \ref{shortprec} implies that $\aut(\U^\prec)$ is extremely amenable. 
    
    It follows from Lemma~4.2 in Malicki \cite{Mal} (applied to $X=\U_\Q$) that for any $Y\in\UltIso^{\Qpos}_\fin$ there is $Y \leq Z\in\UltIso^{\Qpos}_\fin$ such that any partial isometry of $Y$ extends to an isometry of $Z$ (i.e. the domain and range are equal to $Z$). Then Proposition~6.4 in Kechris--Rosendal \cite{KR} implies that $\Iso(\U_\Q)$ contains a countable chain of compact subgroups whose union is dense in $\Iso(\U_\Q)$. 
    As compact groups are amenable, we immediately conclude 
    that $\Iso(\U_\Q)=\Iso(\U)$ is amenable. Therefore by Theorem~\ref{short} $\aut(\U)$ is amenable.
\end{proof}

\begin{uwgi}
The property mentioned in the proof above is called the \emph{extension property for partial automorphisms} (EPPA), see \cite{HerLas} and \cite{Hru}. In the special case when $\mathcal{C}$ is a Fraïssé class it says that
 for any $Y\in\mathcal{C}$ there is $Y \leq Z\in\mathcal{C}$ such that any partial automorphism of $Y$ extends to an automorphism of $Z$.
We note that another difference between the isometric and dc categories is that $\UltIso^{\Qpos}_\fin$ has EPPA as shown by Malicki, while $\Ult_\fin$ does not: no partial dc-isomorphism that moves a distance can be extended to a total dc-automorphism of a finite space.
\end{uwgi}

Using the Kechris--Pestov--Todorčević correspondence~\cite[Theorem~4.8]{KPT} translating between extreme amenability and the Ramsey property, we immediately obtain the following corollary.
\begin{wn}
    $\UltConv_\fin$ is a Ramsey class.
\end{wn}

Using the general theory developed by Kechris–Pestov–Todorčević \cite{KPT} we can strengthen Theorem \ref{examen}(1) by explicitly describing the universal minimal flow of $\aut(\U)$.
Let ${\rm{CLO}}(\U)$ be the space of all convex linear orderings on $\U$. This is a closed subspace of $\{0,1\}^{\U\times \U}$, hence it is compact. $\aut(\U)$ acts continuously on  $\{0,1\}^{\U\times \U}$ via $x\mathrel{(g\cdot {<})}y$ iff $g^{-1}(x)<g^{-1}(y)$.
Let $\prec^*$ denote the generic convex order of $\U^\prec$, i.e. $\U^\prec=(\U,\prec^*)$, and
let $X^*= \overline{\aut(\U)\cdot {\prec^*}}$. Then $X^*= {\rm{CLO}}(\U)$. 

\begin{tw}\label{tw: umf-desrip}
   The universal minimal flow of $\aut(\U)$ is 
   $\aut(\U) \curvearrowright {\rm{CLO}}(\U)$.
\end{tw}

\begin{proof}
We apply Theorem \ref{examen}(2) and work of  Kechris–Pestov–Todorčević \cite{KPT} (Theorems 4.8 and 10.8), which provide a description of the universal minimal flow. The
\cite[Theorem 4.8]{KPT} implies that $\ultraf$ has Ramsey property, needed to apply \cite[Theorem 10.8]{KPT}. We have to check that the expansion $\ultraconf$ of $\ultraf$
is reasonable (for any embedding $f\colon A_0\to B_0$ in $\ultraf$ and an expansion $A\in\ultraconf$ of $A_0$ there is an expansion $B\in\ultraconf$ of $B_0$ such that $f\colon A\to B$ is an embedding in $\ultraconf$), and has the ordering property (for any $A_0\in \ultraf$ there is a $B_0\in \ultraf$ such that for all expansions 
$A\in\ultraconf$ of $A_0$ and  $B\in\ultraconf$ of $B_0$, $A$ embeds into $B$). 

Note that reasonability of the expansion is exactly star-surjectivity of the forgetful functor $F^\prec_\fin\maps \Ult^\prec_\fin \to \Ult_\fin$, which we already proved in Lemma~ \ref{thm:forget_convex}.
As observed in \cite[Section 2.3.2]{Nmem}, the ordering property follows from reasonability and from the fact that every $A_0 \in \Ult_\fin$ can be embedded in an \emph{expansion-invariant} $B_0 \in \Ult_\fin$ (meaning that all expansions $B \in (F^\prec_\fin)^{-1}(B_0)$ are isomorphic).
In $\Ult_\fin$ expansion-invariant spaces correspond to regularly branching leveled meet trees $T \in \LevTree^0_\fin$, i.e. every $t, t' \in T$ on the same level have the same number of successors.
It is easy to see that every $A_0$ isometrically embeds into an expansion-invariant $B_0$ with $D_{B_0} = D_{A_0}$: take $n$ the maximal number of immediate successors appearing in the corresponding tree $\Beebar_{A_0}$ and let $B_0$ be the space whose tree is such that the number of immediate successors of every non-maximal vertex is equal to $n$.
\end{proof}

\subsection{Katětov functors and universality}

While the Fraïssé limit $U$ is universal for the class $\sigma\C_\fin$ of countable structures, its automorphism group $\Aut(U)$ can also be (topologically) universal for the collection of Polish groups $\{\Aut(X) : X \in \sigma\C_\fin\}$.
Furthermore, this universality is sometimes realized via an embedding of $X$ into $U$.
Specifically, an embedding $\phi\maps \Aut(X) \to \Aut(U)$ can be an extension operator with respect to an embedding $e\maps X \to U$: we have $\phi(f) \cmp e = e \cmp f$ for every $f \in \Aut(X)$.
Such extension operators can be obtained using \emph{Katětov functors}.

Let $\C$ be a $\sig$-complete category of structures and suppose that $\C_\fin$ is generated by a family of morphisms $\T \subs \C_\fin$, i.e. every $\C_\fin$-map is a finite composition of $\T$-maps.
Typically, $\T$ consists of all one-point extensions.

A \emph{(small) Katětov functor} is an endofuctor $F\maps \C_\fin \to \C_\fin$ together with a natural transformation $\eta\maps \id{\C_\fin} \to F$ such that for every $\C_\fin$-object $X$, $\eta_X$ absorbs $\T(X, {\to})$, i.e. for every $\T$-map $f\maps X \to Y$ there is a $\C_\fin$-map $g\maps Y \to F(X)$ such that $g \cmp f = \eta_X$.

If $\C_\fin$ has a Katětov functor $F$, then $\C_\fin$ has the amalgamation property, and (if $\C$ is also directed) $X \to^{\eta_X} F(X) \to^{\eta_{F(X)}} F^2(X) \to \cdots$ is a Fraïssé sequence for every $\C_\fin$-object $X$.
Moreover, the colimit $U = F^\omega(X) \in \sig\C_\fin$ is the Fraïssé limit, and for every $\C_\fin$-map $f\maps X \to Y$ we have a functorial $\sig\C_\fin$-extension $F^\omega(f)\maps F^\omega(X) \to F^\omega(Y)$.
In particular, for every $\C_\fin$-object $Y$ we have a (topological) group embedding $\Aut(Y) \to \Aut(F^\omega(Y)) \cong \Aut(U)$.
Moreover, $F$ admits an essentially unique $\sig$-continuous extension $\sig\C_\fin \to \sig\C_\fin$, and we can obtain $U$ also as the colimit of $X \to^{\eta_X} F(X) \to^{\eta_{F(X)}} F^2(X) \to \cdots$ for every $\sig\C_\fin$-object $X$.
Similarly, for every $\sig\C_\fin$-object $Y$ we have a (topological) group embedding $\Aut(Y) \to \Aut(U)$.
For more details see \cite{KubMas}.

\begin{ex}
There is a Katětov functor $F_\Lin\maps \Lin_\fin \to \Lin_\fin$.
It is easy to see that for every linear order $X = \set{x_0 < x_1 < \cdots < x_{k - 1}}$ the inclusion $\eta_X\maps X \to F_\Lin(X) := \set{x'_0 < x_0 < x'_1 < x_1 < \cdots < x'_{k - 1} < x_{k - 1} < x'_k}$ absorbs every one-point extension of $X$.
For an embedding $f\maps X \to Y$ we put $F_\Lin(f)(x_i) = f(x_i) =: y_{g(i)}$ so that $F_\Lin(f) \cmp \eta_X = \eta_Y \cmp f$, and $F_\Lin(f)(x'_i) := y'_{g(i)} = \max\set{y \in Y: y < x$ for every $x > x'_i}$ for $i < k$ and $F_\Lin(f)(x'_k) := \max(Y)$.
This definition turns $F_\Lin$ into a functor.

Similarly we obtain a Katětov functor $F_{\Lin^0}\maps \Lin^0_\fin \to \Lin^0_\fin$.
\end{ex}

Next we consider the category $\Ult$ of two-sorted ultrametric spaces and dc-embeddings.
For $\T$ we take the class of all one-point extensions of finite spaces, i.e. of all dc-embeddings that either add a new distance while keeping the same set of points, or add a new point without introducing new distances.
Clearly, $\T$ generates $\Ult_\fin$.

\begin{prop} \label{thm:Katetov}
	There is a Katětov functor $F\maps \Ult_\fin \to \Ult_\fin$.
\end{prop}
\begin{proof}
    We proceed in two steps.
    First, we consider a functor $G\maps \Ult_\fin \to \Ult_\fin$ accomodating all one-point extensions of the distance sets, i.e. for every finite space $\seq{X, D_X}$ we put $G(\seq{X, D_X})) = \seq{X, L(D_X)}$ and $\eta^G_X := \seq{\id{X}, \eta^L_{D_X}}\maps \seq{X, D_X} \to \seq{X, L(D_X)}$ where $L$ is a Katětov functor $\Lin^0_\fin \to \Lin^0_\fin$.
    For an $\Ult_\fin$-map $f\maps X \to Y$ we put $G(f) := \seq{f, L(D_f)}\maps \seq{X, L(D_X)} \to \seq{Y, L(D_Y)}$.
    It is easy to see that $G$ is a functor $\Ult_\fin \to \Ult_\fin$ and that $\eta^G\maps \id{\Ult_\fin} \to G$ is a natural transformation.
    
    Next, we define a functor $H\maps \Ult_\fin \to \Ult_\fin$ accommodating all one-point extensions of the space while not changing the distance set.
    Let $X$ be a finite space.
    For every $\seq{x, r}, \seq{x', r'} \in X \times D_X$ we put
    \[
        d(\seq{x, r}, \seq{x', r'}) := \begin{cases}
            0 & \text{if $r = r'$ and $d_X(x, x') \leq r$,} \\
            \max\set{d_X(x, x'), r, r'} > 0 & \text{otherwise.}
        \end{cases}
    \]
    Intutively, a pair $\seq{x, r}$ represents a one-point extension with the closest point $x$ at distance $r$.
    We show that $d$ is a $D_X$-valued pseudo-ultrametric on $X \times D_X$.
    Clearly, $d$ is symmetric and $d(\seq{x, r}, \seq{x, r}) = 0$.
    We show \[
        q := d(\seq{x, r}, \seq{x'', r''}) \leq \max\set{d(\seq{x, r}, \seq{x', r'}), d(\seq{x', r'}, \seq{x'', r''})} =: \max\set{q_1, q_2}.
    \]
    If $q_1 = 0 = q_2$, we have $r = r' = r''$ and $d(x, x'') \leq \max\set{d(x, x'), d(x', x'')} \leq r$, and so $q = 0$.
    If $q_1 \neq 0 \neq q_2$, we have 
    \[
        q \leq \max\set{d(x, x''), r, r''} \leq \max\set{d(x, x'), d(x', x''), r, r', r''} = \max\set{q_1, q_2}.
    \]
    If without loss of generality $q_1 = 0 \neq q_2$, then $d(x, x') \leq r = r'$ and so 
    \[
        q \leq \max\set{d(x, x''), r, r''} \leq \max\set{d(x, x'), d(x', x''), r, r''} = \max\set{d(x', x''), r', r''} = q_2.
    \]
    
    Let $X' := (X \times D_X)/{\sim_d}$ be the induced $D_X$-valued ultrametric space.
    We put $H(X) := X'$, and for any $\Ult_\fin$-map $f\maps X \to Y$ we put $H(f)(\seq{x, r}) := \seq{f(x), D_f(r)}$.
    This is a valid definition of a functor $H\maps \Ult_\fin \to \Ult_\fin$.
    We have 
    \begin{gather*}
        d_X(x, x') \leq r = r' \text{ if and only if }d_Y(f(x), f(x')) \leq D_f(r) = D_f(r'), \\
        D_f(\max\set{d_X(x, x'), r, r'}) = \max\set{d_Y(f(x), f(x')), D_f(r), D_f(r')}.
    \end{gather*}
    Therefore, $H(f)$ is well-defined on the pseudo-ultrametric spaces $X \times D_X \to Y \times D_Y$, and hence induces a dc-embedding $X' \to Y'$ with $D_{H(f)} = D_f$.
    Moreover, we have an isometric embedding $\eta^H_X\maps X \to X'$ defined by $\eta^H_X(x) = \seq{x, 0}$, and these embeddings together form a natural transformation $\eta^H\maps \id{\Ult_\fin} \to H$.
    
    Finally, we put $F := H \cmp G$ and $\eta^F := \eta^H_G \cmp \eta^G$ (i.e. the horizontal composition of $\eta^H$ and $\eta^G$).
    It remains to show that $F\maps \Ult_\fin \to \Ult_\fin$ is a Katětov functor.
    Consider a finite ultrametric space and its one-point extension (without loss of generality an inclusion) $i\maps X \to Y$.
    We have either $Y = X$ and $D_Y = D_X \cup \set{q}$, or $Y = X \cup \set{y}$ and $D_Y = D_X$.
    In the first case, there is an $\Lin^0_\fin$-map $g\maps D_Y \to L(D_X)$ such that $g \cmp D_i = \eta^L_{D_X}$ since $L$ is a Katětov functor $\Lin^0_\fin \to \Lin^0_\fin$, and so for $f = \seq{\id{X}, g}$ we have $f \cmp i = \eta^G_X$.
    Hence, $\eta^H_{G(X)} \cmp f \cmp i = \eta^F_X$.
    
    In the second case, let $r_0 := \min\set{d_Y(x, y): x \in X} \in D_X \setminus \set{0}$ and let $x_0 \in X$ be any point such that $d_Y(x_0, y) = r_0$.
    Then $f\maps Y \to F(X)$ defined by $f(x) := \eta^F_X(x)$ for $x \in X$, $D_f(r) := \eta^F_X(r)$ for $r \in D_X$, and $f(y) := \seq{x_0, L(r_0)} \in H(G(X))$ clearly satisfies $f \cmp i = \eta^F_X$.
    It remains to show that $f$ is a dc-embedding.
    It is enough to check that $q_1 := d_{F(X)}(f(x), f(y)) = D_f(d_Y(x, y)) =: q_2$ for every $x \in X$.
    We have 
    \[
        q_1 = d_{F(X)}(\seq{x, 0}, \seq{x_0, L(r_0)}) = \max\set{d_{G(X)}(x, x_0), L(r_0)} = L(\max\set{d_X(x, x_0), r_0}).
    \]
    By the definition of $x_0$ and $r_0$ and by the ultrametric triangle inequality, we have $d_Y(x, y) = \max\set{d_X(x, x_0), d_X(x_0, y)}$, and hence $q_2 = D_f(\max\set{d_X(x, x_0), r_0)}) = q_1$.
\end{proof}

\begin{wn}\label{wn-katetov}
     There is a topological embedding $\Aut(X) \to \Aut(\U)$ for every countable ultrametric space $X$.
\end{wn}

The analogous statement for the completion remains open.
\begin{question} \label{que:complete_aut_universal}
    Is there a topological embedding $\AutM(X) \to \AutM(\clo{\U})$ for every separable ultrametric space $X$?
\end{question}


\end{document}